\documentclass[10pt,reqno]{amsart}
\usepackage[margin=1in]{geometry}
\makeatletter
\def\section{\@startsection{section}{1}%
	\z@{.7\linespacing\@plus\linespacing}{.5\linespacing}%
	{\bfseries
		\centering
}}
\def\@secnumfont{\bfseries}
\makeatother
\usepackage{amsmath,amssymb,amsthm,graphicx,amsxtra, setspace}
\usepackage[utf8]{inputenc}
\usepackage{charter}
\usepackage{mathrsfs}
\usepackage{alltt}
\usepackage{relsize}
\usepackage{hyperref}
\usepackage{aliascnt}
\usepackage{tikz}
\usepackage{mathtools}
\usepackage{multicol}
\usepackage{upgreek}
\usepackage{graphicx,type1cm,eso-pic,color}
\allowdisplaybreaks
\usepackage{scalerel,stackengine}
\stackMath
\newcommand\reallywidehat[1]{%
	\savestack{\tmpbox}{\stretchto{%
			\scaleto{%
				\scalerel*[\widthof{\ensuremath{#1}}]{\kern-.6pt\bigwedge\kern-.6pt}%
				{\rule[-\textheight/2]{1ex}{\textheight}}
			}{\textheight}%
		}{0.5ex}}%
	\stackon[1pt]{#1}{\tmpbox}%
}
\numberwithin{equation}{section}

\newtheorem{theorem}{Theorem}[section]

\newaliascnt{lemma}{theorem}

\aliascntresetthe{lemma} 

\newaliascnt{proposition}{theorem}
\newtheorem{proposition}[proposition]{Proposition}
\aliascntresetthe{proposition}

\newaliascnt{assumption}{theorem}
\newtheorem{assumption}[assumption]{Assumption}
\aliascntresetthe{assumption}

\newaliascnt{corollary}{theorem}
\newtheorem{corollary}[corollary]{Corollary}
\aliascntresetthe{corollary}

\newaliascnt{definition}{theorem}
\newtheorem{definition}[definition]{Definition}
\aliascntresetthe{definition}

\newaliascnt{example}{theorem}

\aliascntresetthe{example}

\newaliascnt{remark}{theorem}
\newtheorem{remark}[remark]{Remark}
\aliascntresetthe{remark}

\newaliascnt{hypothesis}{theorem}

\aliascntresetthe{hypothesis}

\newaliascnt{property}{theorem}

\aliascntresetthe{property}

\let\originalleft\left
\let\originalright\right
\renewcommand{\left}{\mathopen{}\mathclose\bgroup\originalleft}
\renewcommand{\right}{\aftergroup\egroup\originalright}

\makeatletter
\newcommand{\doublewidetilde}[1]{{%
		\mathpalette\double@widetilde{#1}%
}}
\newcommand{\double@widetilde}[2]{%
	\sbox\z@{$\m@th#1\widetilde{#2}$}%
	\ht\z@=.9\ht\z@
	\widetilde{\box\z@}%
}
\makeatother

\renewcommand{\d}{\/\mathrm{d}\/}

\def\w{\textbf{W}^{\varepsilon}_{{\theta}^{\varepsilon}}}

\def\L{\mathrm{L}}

\def\F{\mathrm{F}}
\def\C{\mathrm{C}}

\def\h{\mathbf{h}}
\def\J{\mathrm{J}}

\def\D{\mathrm{D}}
\def\y{\mathbf{y}}

\def\X{\mathbb{X}}
\def\x{\mathbf{x}}
\def\z{\mathbf{z}}
\def\v{\mathbf{v}}
\def\V{\mathbb{V}}
\def\w{\mathbf{w}}
\def\W{\mathrm{W}}
\def\G{\mathbb{G}}

\def\no{\nonumber}
\def\V{\mathbb{V}}

\def\U{\mathrm{U}}

\def\u{\mathbf{u}}
\def\H{\mathbb{H}}
\def\n{\mathbf{n}}
\def\p{\mathbf{p}}

\newcommand{\R}{\mathbb{R}}

\renewcommand{\d}{\/\mathrm{d}\/}

\newcommand{\Addresses}{{
		\footnote{

			\noindent \textsuperscript{1}School of Mathematics,
			Indian Institute of Science Education and Research, Thiruvananthapuram (IISER-TVM),
			Maruthamala PO, Vithura, Thiruvananthapuram, Kerala, 695 551, INDIA.  \par\nopagebreak \noindent
			\textit{e-mail:} \texttt{manikabag19@iisertvm.ac.in}

			\noindent \textsuperscript{2}Department of Mathematics, Ulsan National Institute of Science $\&$ Technology (UNIST), 50 UNIST-gil, Eonyang-eup, Ulju-gun, Ulsan, South Korea \par\nopagebreak
			\noindent  \textit{e-mail:} \texttt{tanibsw91@gmail.com}
			
			\noindent \textsuperscript{3}School of Mathematics, Indian Institute of Science Education and Research, Thiruvananthapuram (IISER-TVM),
			Maruthamala PO, Vithura, Thiruvananthapuram, Kerala, 695 551, INDIA  \par\nopagebreak \noindent
			\textit{e-mail:} \texttt{sheetal@iisertvm.ac.in}

			\noindent \textsuperscript{*}Corresponding author.

			\medskip\noindent
			{\bf Acknowledgments:}  Manika Bag  would like to thank the  Indian Institute of Science Education and Research, Thiruvananthapuram, for providing financial support and the stimulating environment for the research. The work of Tania Biswas was supported by the National Research Foundation of Korea(NRF) grant funded by the Korea government(MSIT)(NO.2020R1A4A1018190). The work of Sheetal Dharmatti is supported by SERB  grant SERB-CRG/2021/008278 of the Government of India.
			
}}}

\begin{document}

	\title[Optimal boundary control for the Cahn-Hilliard-Navier-Stokes Equations]{Optimal boundary control for the Cahn-Hilliard-Navier-Stokes Equations	\Addresses	}

	\author[Manika Bag, T. Biswas and S. Dharmatti]
	{Manika Bag\textsuperscript{1}, Tania Biswas\textsuperscript{2} and Sheetal Dharmatti\textsuperscript{3*}} 
\maketitle

\begin{abstract}
In this work, we study an optimal boundary control problem for a Cahn–Hilliard– Navier–Stokes(CHNS) system in a two-dimensional bounded domain. The CHNS system consists of a Navier–Stokes equation governing the fluid velocity field coupled with a convective Cahn–Hilliard equation for the relative concentration of the fluid. An optimal control problem is formulated as the minimization of a cost functional subject to the controlled CHNS system where the control acts on the boundary of the Navier–Stokes equations. We first prove that there exists an optimal boundary control. Then we establish that the control-to-state operator is Fréchet differentiable and derive first-order necessary optimality conditions in terms of a variational inequality involving the adjoint system.
\end{abstract}



\section{Introduction}
Optimal control problems for fluid flows have been a topic of interest for experimenters and designers for many decades. Mathematicians and computational scientists have lately shown a strong interest in them as well because of the beautiful mathematics behind them. It is well known that the Navier-Stokes equations describe the motion of a fluid. When we consider the evolution of an incompressible isothermal mixture of two immiscible fluids, then phase separation will occur. This phenomenon is described by a diffuse interface model, which is also known as Cahn-Hilliard-Navier-Stokes (CHNS) system [see \cite{liu, gurtin, hohenberg1977theory}, for more details]. Optimal Control problems for Navier-Stokes equations have been studied extensively for both the case of distributed control and boundary control. Motivated by the study of optimal control problems for single fluid flow and its application in various engineering fields, atmospheric sciences, etc, we are interested to study the optimal control problem for the binary mixture described by the CHNS system. Optimal control problems associated with the CHNS system have many applications where one is interested in influencing the phase separation
process, such as in the separation of binary alloys, the formation of polymeric membranes, etc. The boundary control of the fluid can be implemented in various ways in terms of blowing and suction etc. However, real-life implementation of distributed control is many times jkmchallenging. Optimal control problems associated with the CHNS system have been studied by several authors during the last decades analytically as well as numerically. In this paper, we are interested in studying a boundary optimal control problems related to the CHNS system. Mathematically, boundary control problems are harder to deal with, specifically obtaining the optimality conditions, as higher regularity of the solution is often required. Also, the mathematical analysis becomes more challenging since the CHNS system is a highly nonlinear coupled system. To the best of our knowledge, most of the analytical work related to optimal control problems for the CHNS system is devoted to the case of distributed control. This work is the first contribution to the analytic study of the boundary control problem for the CHNS system. Throughout this paper, we assume $T$ is given finite final time and $\Omega$ is a bounded domain in $\mathbb{R}^2$ and $\Gamma$ is the sufficiently smooth boundary of $\Omega$ and $\Gamma'$ is an open subset of $\Gamma$. We set 
\begin{align*}
   Q = \Omega \times (0, T), \quad \Sigma = \Gamma \times (0, T), \quad \Sigma' = \Gamma'\times (0, T).
   \end{align*} 


The system under consideration in this work is the following controlled CHNS system where control acts on the boundary of the Navier-Stokes equation  as a time-dependent Dirichlet boundary condition: 
\begin{equation}\label{equ P2}
\left\{
\begin{aligned}
  \varphi_t + \mathbf{u} \cdot \nabla \varphi &= \Delta \mu, \, \, \text{ in } Q, \\
        \mu &= -\Delta\varphi + F'(\varphi), \\
        \mathbf{u}_t - \nu \Delta\mathbf{u} + (\mathbf{u}\cdot \nabla)\mathbf{u} + \nabla \pi &= \mu \nabla \varphi, \, \, \text{ in } Q, \\
        div~\mathbf{u} & = 0, \, \, \text{ in } Q, \\
        \frac{\partial\varphi}{\partial\n} = 0, \,  \frac{\partial\mu}{\partial\mathbf{n}} & = 0, \,\, \text{ on } \Sigma, \\
        \u  &= M \h, \, \, \text{ on } \Sigma,\\ 
        \mathbf{u}(0) = \mathbf{u}_0 ,\,\, \varphi(0) & = \varphi_0, \,\, \text{ in } \Omega.  
\end{aligned}   
\right.
\end{equation} 
Here $\h$ acts as a control on a part of the boundary $\Gamma'\subset \Gamma$ i.e. $\text{supp}(\h) \subset \Gamma'$, $M$ is an operator defined in Section 2, $\u(\x, t)$ is the average velocity of the fluid, and $\varphi(\x, t)$ is the relative concentration of the fluid.  The density is taken as matched density, i.e., constant density, which is equal to 1. Moreover, $\mu$ denotes the chemical potential, $\pi$ denotes the pressure, $\nu$ denotes the viscosity, and $F$ is a double-well potential. Furthermore, $\mu$ is the first variation of the Helmholtz free energy functional
\begin{align}
\mathcal{E}(\varphi) := \int_{\Omega} \left( \frac{1}{2} | \nabla \varphi|^2 + \F(\varphi (x))\right)\, \d x,
\end{align}
where $\F$ is a  double-well potential of the regular type. A typical example of regular $\F$ is
\begin{align}\label{regular}
    \F(s)=(s^2-1)^2, \ s \in \mathbb{R}.
\end{align}

We now discuss some of the works available in the literature for the solvability of system \eqref{equ P2} when $\h =0$. In \cite{boyer}, the authors established the existence and uniqueness of weak solutions and strong solutions in 2 and 3 dimensions in the case of regular potential (for strong solution global in 2D and local in time in 3D). The authors have proved the existence and uniqueness results for the case of singular potential in \cite{abel3, giorgini}. In \cite{gg}, the authors have studied the asymptotic behavior, where they have proved the existence of global and exponential attractors. For the system \eqref{equ P2}, when the time-dependent Dirichlet boundary condition for the Navier-Stokes equation is considered, in \cite{MTS}, we have proved the existence, uniqueness of weak solution and existence of strong solution. Moreover, we refer [\cite{abel2, abel1, abel, ggm} and references therein] for more generalized models by considering non constant viscosity, general density, thermodynamically consistent model, compressible fluids, moving contact lines, etc.

In this work, our aim is to study the optimal boundary control related to the system \eqref{equ P2}. Let us now briefly mention some related literature on optimal control problems. Optimal control problems for Cahn-Hilliard equations have been studied by several mathematicians in [\cite{CH_dynamicbdry, CH_viscous, CS, HW, ZL}, and the references therein] for both the cases of distributed and boundary control. The authors have considered a distributed optimal control problem for Navier-Stokes equations in \cite{Manservisi_NSdcontrol, Manservisi_dcontrol, sri1, sri2}, to name a few. Whereas in \cite{Manservisi_boundarycontrol, Fursikov_bndry, Fursikov_3dbdary, Fursikov_3dbcontl}, the authors have studied boundary optimal control problem for the Navier-Stokes equation. 
In \cite{Fursikov_bndry}, the 2D Navier-Stokes equation on the unbounded domain is considered with control acts on the boundary. In the seminal paper by the same authors \cite{Fursikov_3dbcontl},  the optimal boundary control for the Navier-Stokes system in 3D has been studied. Turning to optimal control problems for the CHNS system, there are a few analytical works available in the literature that we would like to mention. The robust control problem for the local CHNS system is investigated in \cite{robust_control}, and the optimal control with state constraints is considered in \cite{medjo2}. For the nonlocal CHNS system, we mention \cite{tdm, tdm1, chns_regular, fgj} for regular potential and singular potential, respectively. In \cite{tdm}, the authors have investigated an enstrophy minimization problem and also a data assimilation type of problem where control acts as initial data. In all the other works mentioned above, a distributed optimal control problem has been studied in terms of minimizing a standard tracking type cost functional. 
   Regarding the numerical studies, optimal control problems of semi-discrete CHNS system for various cases like distributed and boundary control, with non-smooth Landau-Ginzburg energies and with non-matched fluid densities are studied in \cite{Hm1, HIm2, Hm3}. These works considered the local Cahn–Hilliard–Navier–Stokes equations for their numerical studies. 


 In the context of boundary control, a more relevant problem from the application point of view is when the control $\h$ acts locally on a part of the boundary instead of being applied to the whole boundary. Hence we consider $\h$  such that it has support in open subset $\Gamma'\subset \Gamma$ i.e., $\text{supp}(\h) \subset \Gamma'$.
  To use the well-posedness results stated in \cite{MTS}, one needs to extend $\h $ to the whole of $\Gamma$ such that the extension remains in the appropriate Hilbert space defined in \cite{MTS}.  The standard extension of $\h$ by zero outside $ \Gamma'$  is denoted by $\tilde \h $. However this extension $\tilde \h $ does not preserve the requisite regularity criteria that we need for the control space. We therefore borrow ideas from [\cite{raymond_feedback}, see section 2] and define an operator $M$ such that the extension $M\tilde\h$ satisfies the regularity assumption. We can then replace the condition $\text{supp}(\h) \subset \Gamma'$ by defining $\u = M\tilde\h = M\h$ on $\Sigma$ in \eqref{equ P2}. With the required well-posedness result of \eqref{equ P2} [see \cite{MTS} for details], the optimal control problems associated with the system \eqref{equ P2} can be studied. We restrict ourselves to dimension 2 as the higher regularity of the solution necessary to study the optimal control problem is available only in dimension 2.

\textbf{Problem description:}
In view of above motivation regarding well-posedness of the system we rewrite the system  \eqref{equ P2} and describe the corresponding control problem under consideration as follows:

Let us consider the quadratic cost functional
\begin{align}
    \mathcal{J}(\u, \varphi, \h) = \frac{1}{2}\int_{0}^{T}\|\u - \u_{Q}\|^2 + &\frac{1}{2}\int
    _{0}^{T}\|\varphi - \varphi_{Q}\|^2 + \frac{1}{2}\|\u(T) - \u_{\Omega}\|^2_{\L^2(\Omega)} \no\\ & + \frac{1}{2}\|\varphi(T) - \varphi_{\Omega}\|_{\mathrm{L}^2(\Omega)}^2 + \frac{1}{2}\int_{0}^T\|\h\|_{\L^2(\Gamma')}^2.\label{cost functional}
\end{align}
Where $\u_Q, \, \varphi_Q, \, \u_{\Omega}, \, \varphi_{\Omega}$ are the target functions such that $\u_Q \in \mathrm{L}^2(0, T; \mathbb{L}^2_{div}(\Omega)), \,$ $ \varphi_Q \in \mathrm{L}^2(Q), \, \u_{\Omega} \in \mathbb{L}^2_{div}(\Omega), \, \varphi_{\Omega} \in \mathrm{L}^2(\Omega)$. 
The corresponding optimal control problem can be defined as 
\begin{align*}
   (\textbf{OCP}) \qquad \min_{\h \in \mathcal{U}_{ad}} \mathcal{J}(\u, \varphi, \h),
\end{align*}
subject to
\begin{equation}\label{equ P}
\left\{
\begin{aligned}
  \varphi_t + \mathbf{u} \cdot \nabla \varphi &= \Delta \mu, \, \, \text{ in } Q, \\
        \mu &= -\Delta\varphi + F'(\varphi), \\
        \mathbf{u}_t - \nu \Delta\mathbf{u} + (\mathbf{u}\cdot \nabla)\mathbf{u} + \nabla \pi &= \mu \nabla \varphi, \, \, \text{ in } Q, \\
        div~\mathbf{u} & = 0, \, \, \text{ in } Q, \\
        \frac{\partial\varphi}{\partial\n} = 0, \,  \frac{\partial\mu}{\partial\mathbf{n}} & = 0, \,\, \text{ on } \Sigma, \\
        \u  & = M\tilde\h \, \, \text{ on } \Sigma, \\
        \mathbf{u}(0) = \mathbf{u}_0 ,\,\, \varphi(0) & = \varphi_0, \,\, \text{ in } \Omega.  
\end{aligned}   
\right.
\end{equation}

The motivation of the $(\mathbf{OCP})$ problem  is  to find the best control $\h$ from the set of admissible controls such that the corresponding optimal solution of \eqref{equ P} is as close as possible to the target state. The last term in \eqref{cost functional} is the effort by the control that we have to pay in order to reach the final state. 

The main results of this paper are summarized as follows:
\begin{enumerate}
    \item We establish the existence of an optimal boundary control for the problem $(\mathbf{OCP})$ where the control is taken as a localised control. [See Theorem \eqref{Existence of an optimal control}]
    \item We show that the control to state operator $\mathcal{S}$ is Fr\'echet differentiable between suitable Banach spaces. [See Theorem \eqref{diffrentiability}]
    \item We derive the adjoint system corresponding to state problem \eqref{equ P} and establish the well-posedness of the adjoint system. [See Theorem \eqref{existence adj equ}]
    \item Finally, we derive the first-order necessary optimality condition in terms of a variational inequality involving adjoint states. [See Theorem \eqref{1st order necessary optimality condition}]
\end{enumerate}

The plan of the paper is as follows: in the next section, we give some preliminaries about function spaces and operators used in the sequel. We also recall the well-posedness results required in this work and give a brief sketch of proof that the strong solution depends continuously on the boundary data. This result is crucial in proving the existence of optimal control which is tackled in section 3.
Section 4 is devoted to studying the linearized system, which comes naturally when one wants to establish the differentiability of control to state operator.
Finally, in section 5, we characterise the optimal control with the help of an adjoint system and establish the well-posedness of the adjoint system.

\section{Preliminary}
\subsection{Functional Setup} 
Let $\Omega$ be a bounded subset of $\mathbb{R}^2$ with sufficiently smooth boundary $\partial \Omega$. We introduce the functional spaces that will be useful in the paper. 
\begin{align*}
&\G_{\text{div}} := \Big\{ \u \in \mathrm{L}^2(\Omega;\R^2) : \text{ div }\u=0,\  \u\cdot \mathbf{n}\big|_{\Gamma}=0 \Big\}, \\
&\V_{\text{div}} :=\Big\{\u \in \mathrm{H}^1_0(\Omega;\R^2): \text{ div }\u=0\Big\},\\ 
&\mathbb{L}^2_{\text{div}} :=\Big\{\u \in \mathrm{L}^2(\Omega;\R^2) : \text{ div }\u=0 \Big\},\\
&\H^s_{\text{div}} :=\Big\{\u \in \mathrm{H}^s(\Omega;\R^2): \text{ div }\u=0\Big\}, \quad s > 0,\\ 
&\mathrm{L}^2:=\mathrm{L}^2(\Omega;\mathbb{R}),\\ 
&\mathrm{H}^s:=\mathrm{H}^s(\Omega;\mathbb{R}),
\end{align*}
where for $0< s <1$, we define the fractional Hilbert space in the usual way:
\begin{align*}
    \H^s = \big\{\u\in \mathbb{L}^2(\Omega) : \frac{|\u(\x) -\u(\y)|}{|\x-\y|^{\frac{n}{2}+s}}\in \mathbb{L}^2(\Omega\times\Omega)\big\}
\end{align*}
with the natural norm, $ \|\u\|_{\H^s} = \|\u\|_{\mathbb{L}^2} + [\u]_{s,2},$
where the seminorm is defined as
\begin{align*}
    [\u]_{s,2} = \Big(\int_{\Omega}\int_{\Omega}\frac{|\u(\x)-\u(\y)|^2}{|\x-\y|^{n+2s}} \, dxdy\Big)^\frac{1}{2}.
\end{align*}
If $s\geq 1$ we write $s= l+ \sigma, \, 0<\sigma<1$ and $l$ is integer, then we write $\|\u\|_{\H^s}= \|\u\|_{\H^l} + \|\u\|_{\H^\sigma}.$ 
Also, we define boundary spaces 
\begin{align*}
  \V^s(\Gamma) := \Big\{\h\in\H^s(\Gamma) : \int_{\Gamma} \h\cdot\n = 0 \Big\}, \quad s\geq 0. 
\end{align*}
 With usual convention, the dual space of $\mathrm{H}^s(\Omega)$ is denoted by $\mathrm{H}^{-s}(\Omega)$.  Let us denote $\| \cdot \|$ and $(\cdot, \cdot)$ the norm and the scalar product, respectively, on $\mathbb{L}^2_{\text{div}}$ and $\G_{\text{div}}$. The duality between any Hilbert space $\X$ and its dual $\X'$ will be denoted by $\left<\cdot,\cdot\right>$. We know that $\V_{\text{div}}$ is endowed with the scalar product 
$$(\u,\v)_{\V_{\text{div}} }= (\nabla \u, \nabla \v)=2(\mathrm{D}\u,\mathrm{D}\v),\ \text{ for all }\ \u,\v\in\V_{\text{div}}.$$ The norm on $\V_{\text{div}}$ is given by $\|\u\|_{\V_{\text{div}}}^2:=\int_{\Omega}|\nabla\u(x)|^2\d x=\|\nabla\u\|^2$. Since $\Omega$ is bounded, the embedding of $\mathbb{V}_{\text{div}}\subset\G_{\text{div}}\equiv\G_{\text{div}}'\subset\V_{\text{div}}'$ is compact. 
\subsection{\bf Linear and Nonlinear Operators}
    Let us define the Stokes operator $\mathbf{A} : \D(\mathbf{A})\cap  \G_{\text{div}} \to \G_{\text{div}}$ by 
$$\mathbf{A}=-\mathrm{P} \Delta,\ \D(\mathbf{A})=\mathbb{H}^2(\Omega) \cap \V_{\text{div}},$$ where $\mathrm{P} : \mathbb{L}^2(\Omega) \to \G_{\text{div}}$ is the \emph{Helmholtz-Hodge orthogonal projection}. Note also that we have
\begin{equation*}
\langle\mathbf{A}\u, \v\rangle = (\u, \v)_{\V_\text{div}} = (\nabla\u, \nabla\v),  \text{ for all } \u \in\D(\mathbf{A}), \v \in \V_{\text{div}}.
\end{equation*}
Let $\u, \, \v$ be two vector-valued functions. Then
\begin{equation*}
    \nabla(\u \cdot \v) = (\nabla\u)\cdot\v + (\nabla\v)\cdot\u = \v^T(\nabla\u) + \u^T(\nabla\v).
\end{equation*}
\\
Let $m$ be a function in $C^2(\Gamma)$ such that $0\leq m(x) \leq 1, \, \forall x \in \Gamma$. $m$ has support in $\Gamma'$ and takes value $1$ in $\Gamma''$, which is open subset of $\Gamma'$. Associated to this function $m$ we introduce the operator $M \in \mathcal{L}(\V^0(\Gamma))$ by
\begin{align*}
    M\h(\x) = m(\x)\h(\x) - \frac{m}{\int_{\Gamma} m}\Big(\int_\Gamma
    m\h\cdot\n\Big)\n(\x).
\end{align*}
We want to show that if $\h\in \mathrm{L}^\infty(0, T; \V^{\frac{3}{2}}(\Gamma'))\cap \mathrm{L}^2(0, T; \V^{\frac{5}{2}}(\Gamma'))\cap \mathrm{H}^1(0, T; \V^{\frac{1}{2}}(\Gamma'))$ then $M\tilde\h \in \mathrm{L}^\infty(0, T; \V^{\frac{3}{2}}(\Gamma))\cap \mathrm{L}^2(0, T; \V^{\frac{5}{2}}(\Gamma))\cap \mathrm{H}^1(0, T; \V^{\frac{1}{2}}(\Gamma)) $, where $\tilde{\h}$ is the extension of $\h$ by $0$ outside $\Gamma'$. Indeed 
\begin{align*}
    [M\tilde\h(t)]_{\frac{1}{2}, 2} &:=\Big(\int_{\Gamma}\int_{\Gamma}\frac{|M\tilde\h(\x, t)-M\tilde\h(\y,t)|^2}{|\x-\y|^3}\Big)^{\frac{1}{2}}\\
    & \leq \Big(\int_{\Gamma'}\int_{\Gamma'}\frac{|m(\x)\h(\x,t)-m(\y)\h(\y, t)|^2}{|\x-\y|^3}\Big)^\frac{1}{2} + C\Big(\int_{\Gamma'}\int_{\Gamma'}\frac{|\n(\x)-\n(\y)|^2}{|\x-\y|^3}\Big)^\frac{1}{2}\\
    &\leq \Big(\int_{\Gamma'}\int_{\Gamma'}\frac{|(m(\x)\h(\x,t)-m(\x)\h(\y, t))-(m(\y)-m(\x))\h(\y, t)|^2}{|\x-\y|^3}\Big)^\frac{1}{2}+ C[\n]_{\frac{1}{2}, 2}\\
    & \leq \sup_{\x\in \Gamma'}|m(\x)|[\h]_{\frac{1}{2}, 2} + 2\sup_{\x\in \Gamma'}|m(\x)| \Big(\int_{\Gamma'}\int_{\Gamma'\cap \{|\x-\y|\geq 1\}}\frac{|\h(\y)|^2}{|\x-\y|} +\\& \int_{\Gamma'}\int_{\Gamma'\cap \{|\x-\y| < 1\}}\frac{|\h(\y)|^2}{|\x-\y|} \Big)^\frac{1}{2} + C\\
    & \leq C([\h]_{\frac{1}{2}, 2} + \|\h\|_{\mathbb{L}^2(\Gamma)} +C) < \infty.
\end{align*}

\begin{assumption}\label{prop of F}
We take the following assumptions on $F$:
	\begin{enumerate}
	\item [(1)] There exist $C_1 >0$, $C_2 > 0$  such that $\F''(s) \leq C_1|s|^{p-1} + C_2$, for all $s \in \mathbb{R}$, $1 \leq p < \infty$ and a.e., $x \in \Omega$.
	\item [(2)] $\F \in \C^{2}(\mathbb{R})$ and there exists $C_3 >0$ such that $\F''(s)\geq -C_3$, for all $s \in \mathbb{R}$, a.e., $x \in \Omega$.
	\item [(3)] There exist $C_3' >0$, $C_4 \geq 0$ and $r \in (1,2]$ such that $|\F'(s)|^r \leq C_3'|\F(s)| + C_4,$ for all $s \in \mathbb{R}$.
		\item[(4)] $\F \in \C^{3}(\mathbb{R})$ and there exists $C_5 > 0$, $|\F'''(s)| \leq C_5(1 + |s|^{q})$ for all $s \in \mathbb{R} \text{ where } q < +\infty$.
		\item[(5)] $F(\varphi_0) \in L^1(\Omega)$.
	\end{enumerate}
\end{assumption}
The well-posedness of the system \eqref{equ P} is proved in \cite{MTS}. As we need strong solution of \eqref{equ P} to study the control problem, we state the strong solution existence result from \cite{MTS}.
\subsection{Well-Posedness of Stokes Equations:}
Let us consider the following Stokes' problem
\begin{equation}\label{equ1}
\left\{
 \begin{aligned}
     -\nu\Delta \u_e + \nabla\pi & = 0, \,\, \text{ in }Q, \\
    div~\u_e & = 0, \, \, \text{ in }Q,\\
    \u_e & = M\tilde\h, \, \, \text{ on } \Sigma.
     \end{aligned}   
     \right.
\end{equation}
\begin{theorem} \label{elliptic lifting}
Suppose that $\h$ satisfies the conditions
    \begin{equation}\label{e1}
\left\{
\begin{aligned}
   \h  &\in \mathrm{L}^2(0, T; \V^\frac{5}{2}(\Gamma')) \cap \mathrm{L}^\infty(0, T; \V^{\frac{3}{2}}(\Gamma')), \\
   \partial_t\h &\in \mathrm{L}^2(0, T; \V^{\frac{1}{2}}(\Gamma'))
\end{aligned}   
\right.
\end{equation}
then  Stokes' equations \eqref{equ1} admits a strong solution
\begin{align}\label{e18}
  \u_e \in \mathrm{H}^1(0, T; \H^1_{div})  \cap \L^2(0, T; \H^3_{div}),
\end{align}
such that
\begin{align*}
    \int_{0}^{T} {\| \u_e(t) \|}^2_{\H^3_{div}} \, dt \leq C \int_{0}^{T} {\| \h(t) \|}^2_{\V^\frac{5}{2}(\Gamma')} \, dt, \\
    \int_{0}^{T} {\| \partial_t\u_e(t) \|}^2_{\H^1_{div}} \, dt \leq C \int_{0}^{T} \| \partial_t \h(t) \|_{\V^{\frac{1}{2}}(\Gamma')} \, dt.
\end{align*}
    \end{theorem}
\begin{proof}
    The proof of the above theorem can be found in \cite{Lions_book, raymond_stokes}.
\end{proof}
 \begin{theorem}[\cite{MTS}, Theorem 6.1]\label{strong solution}
    Let $F$ satisfy assumption \eqref{prop of F} and $\h$ satisfy condition \eqref{e1}. Also assume for initial data $(\u_0, \varphi_0)\in\H^1_{div} \times \mathrm{H^2}$ with $\u_0$ satisfying the compatibility condition
    \begin{align}\label{compatibility}
       M\tilde\h|_{t=0} = \u_0|_{\Gamma} 
    \end{align}
    Then there exists a unique pair $(\u, \varphi)$ which is a weak solution of the system \eqref{equ P} and  satisfies
\begin{align*}
    \u & \in \L^\infty(0, T; \H^1_{div}) \cap \L^2(0, T; \H^2_{div})\\
    \varphi & \in \mathrm{L}^\infty(0, T; \mathrm{H}^2) \cap \mathrm{L}^2(0, T; \mathrm{H}^4).
\end{align*}
 \end{theorem}
\begin{remark}     
One can easily estimate the time derivatives $\u_t, \, \varphi_t$. Thus we have
\begin{align*}
    \u_t \in \mathrm{L}^2(0, T; \mathbb{L}^2_{div}(\Omega)), \quad
    \varphi_t \in \mathrm{L}^2(0, T; \mathrm{L}^2(\Omega)).
\end{align*}
    \end{remark}

In the next section, we will discuss the continuous dependence of a strong solution. We will give a brief sketch of proof as it will follow similarly from \cite{MTS}, Section 5 with a slight modification in estimates.
\subsection{Continuous dependence of strong solution}
Let $(\u_1, \varphi_1)$ and $(\u_2, \varphi_2)$ be two weak solutions of the system \eqref{equ P} with non-homogeneous boundaries $M\tilde{\h_1}$ and $M\tilde{\h_2}$ and initial conditions $\u_{i0}$, $\varphi_{i0}$, for $i=1,2$ respectively. Then denote the differences $\u = \overline{\u}_1 - \overline{\u}_2, \, \varphi = \varphi_1 - \varphi_2$ where $\overline{\u}_i=\u_i - \u_{{e_i}}$, $\u_{e_i}$ is the solution of \eqref{equ1} corresponding to boundary $M\tilde{\h_i}$, for $i=1,2$.  Note that, $\u_e := \u_{e_1} - \u_{e_2}$ satisfies the equation \eqref{equ1} with the boundary $M\tilde\h = M\tilde{\h_1} - M\tilde{\h_2}$. Then $(\u, \varphi)$ satisfies:
\begin{equation}\label{equ P1}
\left\{
\begin{aligned}
& \varphi_t + \u \cdot \nabla \varphi_1 + \mathbf{u}_2 \cdot \nabla \varphi + \u_e \cdot \nabla \varphi_1 + \u_{e_2} \cdot \nabla \varphi  = \Delta \Tilde{\mu}, \, \, \text{ in } \Omega \times (0,T),  \\
    & \Tilde{\mu} = -\Delta\varphi + F'(\varphi_2) - F'(\varphi_1), \, \, \text{ in } \Omega \times (0,T), \\
    & \u_t - \nu \Delta\u + (\u\cdot \nabla)\u_{e_1} + (\u_{e_1} \cdot \nabla)\u + (\u\cdot \nabla)\u_1 + (\u_2 \cdot \nabla) \u_e + (\u_e \cdot \nabla)\u_2   \\ &\qquad + (\u_e \cdot \nabla)\u_{e_1} + (\u_{e_2} \cdot \nabla)\u_e + \nabla \Tilde{\pi} = \Tilde{\mu} \nabla \varphi_1 + \mu_2 \nabla \varphi- \partial_t\u_e, \, \, \text{ in } \Omega \times (0,T), \\
   & div~\mathbf{u}  = 0, \, \, \text{ in } \Omega \times (0,T), \\
    &    \frac{\partial\varphi}{\partial\n} = 0, \,  \frac{\partial\Tilde{\mu}}{\partial\mathbf{n}}  = 0, \,\, \text{ on } \Sigma,\\
     &   \u  = 0, \,\, \text{ on } \Sigma,  \\
     & \u(x,0)= \u_0, \, \varphi(x,0)=\varphi_0, \, \text{in } \Omega.
\end{aligned}
\right.
\end{equation}
We have the following continuous dependence result for the strong solution of \eqref{equ P}:
\begin{proposition}\label{cds}
    Let $\h_i \in \mathcal{U} \text{ and } (\varphi_i, \u_i)$ be the strong solution of the system \eqref{equ P} with corresponding boundary data $\h_i$  with initial data $(\varphi_{i0}, \u_{i0})$ for $i = 1, 2.$ Let $\h = \h_1 - \h_2$. Then there exist a constant $C>0$ such that:
    \begin{align}\label{strong dependence}
        \|\u\|_{\mathrm{L}^\infty(0, T; \H^1_{div}(\Omega))\cap\mathrm{L}^2(0, T; \H^2(\Omega))} + \|\varphi\|_{\mathrm{L}^\infty(0, T; \mathrm{H}^2(\Omega))\cap\mathrm{L}^2(0, T; \mathrm{H}^4(\Omega))} + \|\nabla\Tilde{\mu}\|_{\mathrm{L}^2(0, T; \mathrm{L}^2)}\leq C \|\h\|_{\mathcal{U}},
    \end{align}
    where the normed linear space $\mathcal{U}$ is as defined in \eqref{control space} with the norm, $\|\cdot\|_{\mathcal{U}}$, given by \eqref{norm h}.
\end{proposition}
\begin{proof}
We multiply the equation $\eqref{equ P1}_3$ by $\mathbf{A}\u$ and integrating we get the following    
\begin{align}\label{cont u}
  &\frac{1}{2}\frac{d}{dt} \|\nabla\u\|^2 + \nu \|\mathbf{A}\u\|^2 = (\, (\u\cdot \nabla)\u_{e_1}, \mathbf{A}\u \, )  + (\, (\u_{e_1} \cdot \nabla)\u, \mathbf{A} \u \, ) + ( \, (\u\cdot \nabla)\u_1, \mathbf{A} \u \, ) \no \\ & \qquad + ( \, (\u_2 \cdot \nabla) \u_e, \mathbf{A} \u \, ) + ( \, (\u_e \cdot \nabla)\u_2, \mathbf{A} \u \, ) + ( \, (\u_e \cdot \nabla)\u_{e_1}, \mathbf{A} \u \, ) + ( \, (\u_{e_2} \cdot \nabla)\u_e, \mathbf{A} \u \, ) \no \\ & \qquad  + ( \, \Tilde{\mu} \nabla \varphi_1, \mathbf{A} \u \, ) + ( \, \mu_2 \nabla \varphi, \mathbf{A} \u \, ) + ( \, \partial_t\u_e, \mathbf{A} \u \, )
  \end{align}
  Similarly multiplying by $\Delta^2\varphi$ to the equation $\eqref{equ P1}_1$  and integrating we obtain
\begin{align}\label{cont phi}
   \frac{1}{2}\frac{d}{dt} \|\Delta \varphi\|^2 + \|\Delta^2\varphi\|^2  = & -(\u \cdot \nabla \varphi_1, \Delta^2\varphi) - (\u_2 \cdot \nabla \varphi, \Delta^2\varphi) - (\u_e \cdot \nabla \varphi_1, \Delta^2\varphi) \no\\ & - (\u_{e_2} \cdot \nabla \varphi, \Delta^2\varphi) + (\Delta \Tilde{\mu}, \Delta^2\varphi) 
\end{align}
Again to get the $\Tilde{\mu}$ estimate we multiply the equation $\eqref{equ P1}_1$ by $\Tilde{\mu}$ and integrate to obtain
\begin{align}\label{mu}
     \frac{1}{2}\frac{d}{dt}\|\nabla\varphi\|^2 + C_3 \frac{d}{dt}\|\varphi\|^2  + \|\nabla \Tilde{\mu}\|^2 = &-(\u \cdot \nabla \varphi_1, \Tilde{\mu}) - (\u_2 \cdot \nabla\varphi, \Tilde{\mu}) - (\u_e \cdot \nabla \varphi_1, \Tilde{\mu})\no \\ &-(\u_{e_2} \cdot \nabla \varphi,\Tilde{\mu}),
     \end{align}
Now estimating each term of the right-hand side of \eqref{cont u} using Sobolev inequality, H{\"o}lder inequality, Young's inequality and keeping in mind the regularity of $(\varphi_i,\u_i) \text{ and } \u_{ei}$ we obtain
\begin{align}\label{cd u}
     \frac{1}{2}\frac{d}{dt} \|\nabla\u\|^2 + \nu \|\mathbf{A}\u\|^2  \leq & \frac{\nu}{2}\|\mathbf{A}\u\|^2 + C \Big[2\|\nabla\u_2\|^2 + \|\h_1\|^2_{\V^{\frac{3}{2}}(\Gamma')} + \|\h_2\|^2_{\V^{\frac{3}{2}}(\Gamma')} + \|\varphi_1\|^2_{\mathrm{H}^3} + \frac{5}{\nu}\|\mu_2\|^2_{\mathrm{H}^2} \Big]\|\h\|^2_{\mathcal{U}} \no\\
      & + C\big[\|\h_1\|^2_{\V^{\frac{3}{2}}(\Gamma')} + \|\u_1\|^2_{\H^2_{div}}\big]\|\nabla\u\|^2.
\end{align}
Similarly, adding \eqref{cont phi} and \eqref{mu} and estimating right-hand side we get 
\begin{align}\label{cd phi}
     \frac{1}{2}\frac{d}{dt} \|\varphi\|_{\mathrm{H}^2}^2 + \|\Delta^2\varphi\|^2 + \|\nabla\Tilde{\mu}\|^2 & \leq  \frac{1}{2} \|\Delta^2\varphi\|^2 + C \Big[3\|u_2\|^2_{\H^2} + \|\nabla\varphi_1\|^2 + 3\|\h_2\|^2_{\V^{\frac{3}{2}}(\Gamma')}\Big]\|\h\|^2_{\mathcal{U}} \no\\
     & +\Big(\frac{\nu}{4} + C\|\varphi_1\|^4_{\mathrm{H}^1}\Big)\|\nabla\u\|^2 + C (\|\u_2\|^2_{\mathbb{L}^4_{div}}+ \|\u_{e_2}\|^2_{\mathbb{L}^2_{div}} + \|\h\|_{\V^\frac{3}{2}(\Gamma')}+1) \|\varphi\|_{\mathrm{H}^2}^2. 
\end{align}
Adding the inequalities \eqref{cd u}, \eqref{cd phi} then integrating the resulting inequality between $0 \text{ to } t$ and applying Gronwall's lemma we get 
\begin{align}
   \sup_{t \in [0, T)} \big(\|\nabla\u(t)\|^2 + \|\varphi(t)\|_{\mathrm{H}^2}^2\big) + \int_{0}^t (\|\mathbf{A}\u(t)\|^2 + \|\Delta^2\varphi(t)\|^2 + \|\nabla\Tilde{\mu}(t)\|^2 dt \leq C\|\h\|^2_{\mathcal{U}}
\end{align}
 for all $t \in [0, T)$, where $C$ is a positive constant depends on $\|(\varphi_i, \u_i)\|_{\mathcal{V}}, \|\h_i\|_{\mathcal{U}}, \|\varphi_{i0}\|_{\mathrm{H}^2}, \|\u_{i0}\|_{\H^1_{div}}$, $i = 1, 2$, Which immediately gives \eqref{strong dependence}.

\end{proof}
\section{Optimal Control}
We introduce the space
\begin{align}\label{control space}
    \mathcal{U} := \{ \h(\x, t) :\h \in \mathrm{L}^\infty(0, T; \V^{\frac{3}{2}}(\Gamma'))\cap \mathrm{L}^2(0, T; \V^\frac{5}{2}(\Gamma')), \, \partial_t\h \in \mathrm{L}^2(0, T; \V^{\frac{1}{2}}(\Gamma')\}, 
\end{align} 
with the norm given by 
\begin{align}\label{norm h}
    \|\U\|_{\mathcal{U}} = \|\U\|_{\mathrm{L}^\infty(0, T; \V^{\frac{3}{2}}(\Gamma'))} + \| \U \|_{\mathrm{L}^2(0, T; \V^\frac{5}{2}(\Gamma'))} + \|\partial_t\U\|_{\mathrm{L}^2(0, T; \V^{\frac{1}{2}}(\Gamma'))}, \text{ for any } \U \in \mathcal{U}.
\end{align}
\begin{remark}
    If $\Gamma$ is sufficiently smooth boundary of $\Omega$ then from \cite{Taylor} we have the following embedding
    \begin{align}
        \mathrm{L}^\infty(0, T; \V^{\frac{3}{2}}(\Gamma))\cap \mathrm{L}^2(0, T; \V^\frac{5}{2}(\Gamma)) \hookrightarrow C([0, T]; \V^{\frac{1}{2} + \delta}) \hookrightarrow C([0, T] \times \Gamma),
        \end{align}
        for any $0<\delta \leq 1.$ Therefore the compatibility condition $M\tilde\h(\x, 0) = \u_0(\x)|_{\Gamma}$ is well-defined.
\end{remark}
For a sufficiently large positive constant $L$, we define the set of admissible boundary control space as follows:
\begin{align}\label{ad control space}
    \mathcal{U}_{ad} := \{\h \in \mathcal{U} \, | \, \|\h\|_{\mathrm{L}^\infty(0, T; \V^{\frac{3}{2}}(\Gamma'))\cap \mathrm{L}^2(0, T; \V^\frac{5}{2}(\Gamma'))} \leq L, \, \|\partial_t\h\|_{\mathrm{L}^2(0, T; \V^{\frac{1}{2}}(\Gamma')} \leq L \}.
\end{align}
Let us define
\begin{align}
\mathcal{W} := & \big[C([0, T]; \L^2_{div}(\Omega))\cap \mathrm{L}^2(0, T; \H^1_{div}(\Omega))\big] \no \\ & \times \big[C([0, T], \mathrm{H}^1(\Omega)) \cap \mathrm{L}^2(0, T; \mathrm{H}^2(\Omega))\big],\\
    \mathcal{V} :=& \big[C([0, T]; \H^1_{div}(\Omega))\cap \mathrm{L}^2(0, T; \H^2_{div}(\Omega))\cap\mathrm{H}^1(0, T; \L^2_{div}(\Omega))\big] \no \\ & \times \big[C([0, T], \mathrm{H}^2(\Omega)) \cap \mathrm{L}^2(0, T; \mathrm{H}^4(\Omega))\cap \mathrm{H}^1(0, T; \mathrm{H}^1(\Omega))\big],
\end{align}
which denote the space of  global weak and strong solution of the system \eqref{equ P} respectively.\\
Now let us define control to state operator 
\begin{align*}
    \mathcal{S} : \mathcal{U} \rightarrow \mathcal{V} \, \text{ by } \, \h \rightarrow (\u, \varphi),
\end{align*}
where $(\u, \varphi)$ is the strong solution of \eqref{equ P} corresponding to boundary control $\h$ and with initial data $(\u_0, \varphi_0) \in \H^1_{div} \times \mathrm{H}^2.$

Therefore we can now define a reformulated  cost functional
\begin{align}\label{reformulated}
    \Tilde{\mathcal{J}} : \mathcal{U} \rightarrow [0, \infty) \, \text{ by } \, \Tilde{\mathcal{J}}(\h) := \mathcal{J}(\u, \varphi, \h)= \mathcal{J}(\mathcal{S}(\h), \h).
\end{align}
In the following theorem, we prove that the optimal control problem $(\mathbf{OCP})$ has a solution which exhibits the existence of an optimal control.
\begin{theorem}{(Existence of an optimal control)}\label{Existence of an optimal control}
    Let the initial data $(\u_0, \varphi_0) \in \H^1_{div} \times \mathrm{H}^2$ with $\u_0$ satisfying the compatibility condition \eqref{compatibility}, and $F$ satisfy the Assumption \ref{prop of F}. Let the target functionals $\u_Q \in \mathrm{L}^2(0, T; \mathbb{L}^2_{div}(\Omega)), \, \varphi_Q \in \mathrm{L}^2(\Omega \times (0, T)), \, \u_{\Omega} \in \mathbb{L}^2_{div}(\Omega), \, \varphi_{\Omega} \in \mathrm{L}^2(\Omega)$. There exists a control $\h^{\ast} \in \mathcal{U}_{ad}$ such that
\begin{align*}
  \mathcal{J}(\u^\ast, \varphi^\ast, \h^\ast) = \min_{\h\in\mathcal{U}_{ad}}\mathcal{J}(\u, \varphi, \h),  
\end{align*}

where $(\u^{\ast}, \varphi^\ast)$ is the unique solution of the state problem \eqref{equ P} corresponding to boundary control $\h^\ast\in \mathcal{U}_{ad}$. 
\end{theorem}
\begin{proof}
Let $j = \inf_{\h \in \mathcal{U}_{ad}}\mathcal{J}(\u, \varphi, \h)$. As $\mathcal{J}(\u, \varphi, \h) \geq 0 \text { for all } \h \in \mathcal{U}_{ad}$ so $j \geq 0$. Then there exists a sequence $\h_n \in \mathcal{U}_{ad}$ such that
\begin{align*}
    \lim_{n \rightarrow \infty}\mathcal{J}(\u_n, \varphi_n, \h_n) = j
\end{align*}
where $\mathcal{S}(\h_n) = (\u_n, \varphi_n)$. Since $\{\h_n\}$ is bounded, there exist a subsequence still denoted by $\{\h_n\}$ such that 
\begin{align}
    \h_n & \rightarrow \h^\ast \, \text{ weakly in } \mathrm{L}^2(0, T; \V^\frac{5}{2}(\Gamma')),\label{weak control}\\
    \h_n & \rightarrow \h^\ast \, \text{ weak* in } \mathrm{L}^\infty(0, T; \V^{\frac{3}{2}}(\Gamma')),\label{weak* control}\\
    \partial_t\h_n & \rightarrow \partial_t\h^\ast \, \text{ weakly in } \mathrm{L}^2(0, T; \V^{\frac{1}{2}}(\Gamma'),\label{time weak control}
\end{align}
for some $\h^\ast \in \mathcal{U}$. Now we show that $\h^\ast \in \mathcal{U}_{ad}$. To do that we only need to show $\|\h^\ast\|_{\mathcal{U}} \leq L$. We know that if $\h_n \rightarrow \h^\ast$ weakly in $\| . \|$, then $\|\h^\ast\| \leq \liminf_{n\rightarrow\infty} \|\h_n\|$. Therefore from \eqref{weak control} and \eqref{time weak control} we have
\begin{align*}
    \|\h^\ast\|_{\mathrm{L}^2(0, T; \V^{\frac{5}{2}}(\Gamma'))} \leq L,\\
    \|\partial_t\h^\ast\|_{\mathrm{L}^2(0, T; \V^{\frac{1}{2}}(\Gamma')}) \leq L.
\end{align*}
Again, using the fact that any norm in a Banach space is weak* - lower semicontinuous,  we have 
\begin{align*}
    \|\h^\ast\|_{\mathrm{L}^\infty(0, T; \V^{\frac{3}{2}}(\Gamma'))} \leq L.
\end{align*}
Now from continuous dependence of strong solution, Proposition \ref{cds}, we conclude that there exists a subsequence $(\u_n, \varphi_n)$ such that
\begin{align*}
    \u_n &\rightarrow \u^\ast \, \text{ weak }^\ast \text{ in }  \mathrm{L}^\infty(0, T; \H^1_{\text{div}}),\\
    \u_n &\rightarrow \u^\ast \, \text{ weakly in } \mathrm{L}^2(0, T; \H^2_{\text{div}}) ,\\
    \u_n &\rightarrow \u^\ast \, \text{ strongly in } \mathrm{L}^2(0, T;\mathbb{L}^2_{\text{div}}),\\
    \varphi_n &\rightarrow \varphi^\ast \, \text{ weak }^\ast \text{ in } \,  \mathrm{L}^\infty(0, T; \mathrm{H}^2),\\
    \varphi_n &\rightarrow \varphi^\ast \, \text{ weakly in } \, \mathrm{L}^2(0, T; \mathrm{H}^4),\\
     \varphi_n &\rightarrow \varphi^\ast \, \text{ strongly in } \, \mathrm{L}^2(0, T; \mathrm{H}^1).
\end{align*}
Using the above convergence result on $(\u_n, \varphi_n)$, we can pass to the limit in the weak formulation of the system \eqref{equ P} and the limit $(\u^\ast, \varphi^\ast)$ satisfy the weak formulation of \eqref{equ P} with initial condition $(\u_0, \varphi_0) \in \H^1_{div} \times \mathrm{H}^2,$ and boundary condition $\u^\ast = \h^\ast$. Since $(\u^\ast, \varphi^\ast) \in \mathcal{V} \text{ and } \h^\ast \in \mathcal{U}$, therefore $\mathcal{S}(\h^\ast) = (\u^\ast, \varphi^\ast).$\\
Since $\mathcal{J}$ is weakly lower semi-continuous in $\mathcal{V}\times\mathcal{U}$, we have 
\begin{align*}
    \mathcal{J}(\u^\ast, \varphi^\ast, \h^\ast) \leq \liminf_{n \rightarrow \infty}\mathcal{J}(\u_n, \varphi_n, \h_n).
\end{align*}
We have from the above convergence 
\begin{align*}
    \liminf_{n \rightarrow \infty}\mathcal{J}(\u_n, \varphi_n, \h_n) = j
\end{align*}
Therefore, $j \leq \mathcal{J}(\u^\ast, \varphi^\ast, \h^\ast) \leq  \liminf_{n \rightarrow \infty}\mathcal{J}(\u_n, \varphi_n, \h_n) = j$\\
Hence we conclude that,
\begin{align*}
  \mathcal{J}(\u^\ast, \varphi^\ast, \h^\ast) = \min_{\h\in\mathcal{U}_{ad}}\mathcal{J}(\u, \varphi, \h),  
\end{align*}
which yields that $(\u^\ast, \varphi^\ast, \h^\ast)$ is a solution of optimal control problem.
\end{proof}
The $\h^\ast $ obtained in the Theorem \ref{Existence of an optimal control} is called an optimal control and the corresponding solution 
$(\u^\ast, \varphi^\ast, \h^\ast) $ is called the optimal solution.

\section{Linearized system}
In this section, we will derive the linearized system and establish the existence and uniqueness results for its weak solutions. Then we will show that the Fre\`chet derivative of the control to state operator satisfies the linearized system.

For a fixed control $\hat{\h} \in \mathcal{U}$, let, $\mathcal{S}(\hat{\h}) = (\hat{\u}, \hat{\varphi}) \in \mathcal{V}$ be the strong solution of the system \eqref{equ P} corresponding to boundary data $M\Tilde{\hat{\h}}$. Clearly $\hat{\h}$ has to satisfy \eqref{compatibility} that is $M\Tilde{\hat{\h}}|_{t=0}=\u_0|_{\Gamma}$. Let a control $\eta \in \mathcal{U}-\{\hat{\h}\}$ be given. To prove Fr\'echet differentiability of control to state operator, we consider the following system, which we obtain by linearising the state problem \eqref{equ P} around $(\hat{\u}, \hat{\varphi}).$ In this process we take $\u = \hat{\u} + \w, \, \varphi = \hat{\varphi} + \psi, \, \pi = \hat{\pi} + \overline{\pi}$. Substituting this, we obtain the following linearized system as 
\begin{equation}\label{linearise equ}
\left\{
\begin{aligned}
  &  \w_t - \nu \Delta\w + (\hat{\u} \cdot \nabla) \w + (\w \cdot \nabla)\hat{\u} + \nabla\overline{\pi} = - \Delta\hat{\varphi}\nabla{\psi} - \Delta{\psi}\nabla{\hat{\varphi}} \\
  & + F'(\hat{\varphi})\nabla\psi+ F''(\hat{\varphi})\psi\nabla\hat{\varphi}, \, \text{ in }Q\\
    & \text{div}~\w = 0,\, \text{ in }Q,\\
    & \psi_t + \w \cdot \nabla \hat{\varphi} + \hat{\u} \cdot \nabla \psi = \Delta \mu_{\psi}, \, \text{ in } Q,\\
    & \mu_{\psi} = - \Delta \psi + F''(\hat{\varphi})\psi, \, \text{ in }Q,\\
    & \w = M\Tilde{\mathbf{\eta}}, \, \text{ on } \Sigma,\\
    & \frac{\partial \psi}{\partial n} = 0 , \, \frac{\partial \mu}{\partial n} = 0,\Sigma,\\
    &\w|_{t=0} = 0, \psi|_{t=0} = 0, \text{ in }\Omega.
\end{aligned}
\right.
\end{equation}
\subsection{Existence of a weak solution of linearized system}
Now we state the existence result of the linearized system \eqref{linearise equ} and prove it by using a Faedo-Galerkin approximation scheme.
\begin{theorem}\label{lin equ}
Let $F$ satisfy the Assumption \ref{prop of F}. Then for any $\eta \in \mathcal{U} - \{\hat{\h}\}$ satisfying $M\Tilde{\eta}|_{t=0}= 0$, the system \eqref{linearise equ} admits a unique weak solution $(\w, \psi)$ such that
\begin{align*}
    \w &\in \mathrm{L}^\infty(0, T; \mathbb{L}^2_{div}(\Omega)) \cap \mathrm{L}^2(0, T; \H^1_{div}(\Omega)) \cap \mathrm{H}^1(0, T; \V'_{div}(\Omega)),\\
    \psi&\in \mathrm{L}^\infty(0, T; \mathrm{H}^1(\Omega)) \cap \mathrm{L}^2(0, T; \mathrm{H}^2(\Omega)) \cap \mathrm{H}^1(0, T; (\mathrm{H}^1)'(\Omega)).
\end{align*}
\end{theorem}
\begin{proof}
We prove the theorem using the Faedo-Galerkin approximation scheme. Let us consider the families of functions $(\u_k) \text{ and } (\gamma_k)$, eigenfunctions of the Stokes operator and the Neumann operator $-\Delta+I$ respectively. We consider n-dimensional subspaces $\U_n := \langle u_1, \cdots, \u_n\rangle$ and  $\Psi_n = \langle \psi_1, \cdots,\psi_n \rangle$, spanned by $n$-eigenfunctions, and orthogonal projection on these spaces, $P_n :=P_{\U_n} \text{ and } \overline{P}_n = P_{\Psi_n}$. 
We look for the functions 
\begin{align*}
    &\w_n(t, \x) = \overline{\w}_n(t, \x) + \w_e(t, \x) = \sum_{i=0}^{i=n} a_i(t)\u_i(\x) + \w_e(t, \x),\\
    &\psi_n(t, \x) = \sum_{i=0}^{i=n} b_i(t) \gamma_i(\x),
\end{align*}
which solves the following approximated problem a.e. in $[0, T]$ and for $i = 1, . . . , n$
\begin{align}
    &\langle \partial_t \overline{\w}_n, \u_i \rangle + (\nabla\overline{\w}_n \cdot\nabla \u_i) +( (\hat{\u} \cdot \nabla) (\overline{\w}_n + \w_e), \u_i)  + (((\overline{\w}_n + \w_e) \cdot\nabla) \hat{\u} , \u_i) \no \\ & = -(\Delta\hat{\varphi}\nabla{\psi}_n, \u_i) - (\Delta{\psi}_n\nabla{\hat{\varphi}}, \u_i) +(\psi_n F''(\hat{\varphi})\nabla \hat{\varphi}, \u_i) -\langle\partial_t \w_e, \u_i\rangle  \, \text{ for all } \u_i \in \U_n,  \, \text{ in } \Omega, \label{var w}\,\\
    &\langle \partial_t \psi_n, \gamma_i \rangle +((\overline{\w}_n + \w_e) \cdot \nabla \hat{\varphi} , \gamma_i) + (\hat{\u} \cdot \nabla \psi_n) , \gamma_i) =( \Delta\mu_n, \gamma_i), \, \text{ for all } \gamma_i \in \Psi_n,  \, \text{ in } \Omega, \label{var psi}\\
    & \overline{\w}_n = 0, \quad \frac{\partial\psi}{\partial\n} = 0, \, \text{ on } \Gamma,\\
    &  \overline{\w}_n \big|_{t=0} = 0 , \, \psi_n\big|_{t=0} = 0 \, \text{ in } \Omega,
    \end{align}
    where $\w_e$ is a solution of \eqref{equ1} with $\w_e(t) = M \tilde{\eta}(t)$ on $\Gamma$ and 
 a.e $t \in [0, T]$.\\
    This is a Cauchy problem of the system of $2n$ ordinary differential equations with $2n$ unknowns $a_i, \, b_i$, which can be solved by the Cauchy-Lipschitz theorem. So we get a unique solution $(\overline{\w}_n, \psi_n)$ of approximated system \eqref{var w}-\eqref{var psi}.
   Now we prove some a priori estimates for approximated solution $(\overline{\w}_n, \psi_n)$ independent of $n$. Let us take $\u_i=\overline{\w}_n$ in  \eqref{var w} and $\gamma_i = \psi_n - \Delta\psi_n$ in \eqref{var psi}, and multiply the equation $\eqref{linearise equ}_4$ with $-\Delta\mu_{n} + \Delta\psi_n$, and adding these equations we get
    \begin{align}
      & \frac{1}{2}\frac{d}{dt}(\|\overline{\w}_n\|^2 + \|\nabla\psi_n\|^2 + \|\psi_n\|^2)  + \nu \|\nabla\overline{\w}_n\|^2 + \|\Delta\psi_n\|^2 + \|\nabla \mu_n\|^2 \no \\ & \leq \int_\Omega [(\hat{\u} \cdot \nabla (\overline{\w}_n + \w_e)) \, \overline{\w}_n] dx + \int_\Omega [(\overline{\w}_n + \w_e) \cdot \nabla \hat{\u}) \, \overline{\w}_n] dx  -\int_\Omega \Delta\hat{\varphi}(\nabla{\psi}_n\cdot \overline{\w}_n) dx \no\\&-\int_\Omega\Delta{\psi}_n(\nabla{\hat{\varphi}} \cdot \overline{\w}_n) dx   + \int_\omega \psi_n F''(\hat{\varphi})(\nabla \hat{\varphi} \cdot \overline{\w}_n) dx - \int_\Omega \partial_t \w_e \cdot \overline{\w}_n dx +  \int_\Omega ((\overline{\w}_n + \w_e) \cdot \nabla \hat{\varphi}) \, \Delta \psi_n dx\no\\& + \int_\Omega(\hat{\u} \cdot \nabla \psi_n) \, \Delta \psi_n dx  - \int_{\Omega} F''(\hat{\varphi}) \psi_n \Delta\mu_n dx + \int_{\Omega} F''(\hat{\varphi}) \psi_n \Delta\psi_n dx - \int_{\Omega} ((\overline{\w}_n + \w_e) \cdot \nabla \hat{\varphi}) \, \psi_n \label{linearise approx}\no\\
      & \quad=\sum_{j=1}^{15}K_j
    \end{align}
    We estimate each $K_1 \cdots K_{15}$ on the right-hand side of the above inequality individually. We will use the following abbreviation in the rest of the proof: C is a generic constant that may depend on the norm of $(\hat{\u}, \hat{\varphi})$, $\eta$ but not on $n$. Using H{\"o}lder inequality, Young's inequality, and Agmon's inequality repeatedly gives the following series of estimates:
    \begin{align*}
|K_2| & 
        \leq \|\hat{\u}\|_{\L^4}\|\nabla\overline{\w}_n\| \, \|\w_e\|_{\L^4} \;
         \leq C \|\hat{\u}\|^2_{\H^1_{div}} \|\eta\|^2_{\V^{\frac{3}{2}}(\Gamma')} + \frac{\nu}{8} \|\nabla\overline{\w}_n\|^2,\\
|K_3| & \leq \|\overline{\w}_n\| \, \|\nabla\hat{\u}\|_{\L^4} \|\overline{\w}_n\|_{\L^4} \; \leq C \|\hat{\u}\|^2_{\H^2_{div}} \|\overline{\w}_n\|^2 + \frac{\nu}{8} \|\nabla\overline{\w}_n\|^2,\\
|K_4| & \leq  \|\w_e\|_{\L^4} \|\nabla\overline{\w}_n\| \, \|\hat{\u}\|_{\L^4}  \leq C \|\w_e\|^2_{\H^1_{div}} \|\hat{\u}\|^2_{\H^1_{div}} +  \frac{\nu}{8} \|\nabla\overline{\w}_n\|^2 \no\\
|K_5| &\leq     \|\Delta\hat{\varphi}\|_{\mathrm{L}^\infty}\|\nabla\psi_n\|\|\overline{\w}_n\| \; \leq \frac{1}{2}\|\Delta\hat{\varphi}\|^2_{\mathrm{L}^\infty}\|\nabla\psi_n\|^2 + \frac{1}{2}\|\overline{\w}_n\|,\\
|K_6| & \leq \|\Delta\psi_n\|\|\nabla\hat{\varphi}\|_{\mathrm{L}^\infty}\|\overline{\w}_n\| \;  \leq \frac{3}{4} \|\nabla\hat{\varphi}\|_{\mathrm{L}^\infty}^2\|\overline{\w}_n\|^2 + \frac{1}{12}\|\Delta\psi_n\|^2,\\
|K_7| & = \Big| \int_\Omega F'(\hat{\varphi}) \nabla \psi_n \cdot \overline{\w}_n \, dx \Big|\;  \leq \|F'(\hat{\varphi})\|_{\mathbf{L}^\infty} \|\nabla\psi_n\| \, \|\overline{\w}_n\| \no \\ & \quad \leq C \|\nabla\psi_n\|^2 + \frac{1}{2} \|\overline{\w}_n\|^2,\\
|K_8| & \leq  \frac{1}{2} \|\partial_t \w_e\|^2 + \frac{1}{2} \|\overline{\w}_n\|^2,\\
|K_9| & \leq 
          \frac{1}{12}\|\Delta\psi_n\|^2 + \frac{3}{4} \|\nabla\hat{\varphi}\|^2_{\mathrm{L}^\infty}\|\overline{\w}_n\|^2,\\
|K_{10}| & \leq \frac{1} {12}\|\Delta\psi_n\|^2 + \frac{3}{4}\|\nabla\hat{\varphi}\|^2\|\w_e\|^2,\no\\
|K_{11}| &\leq C  \|\hat{\u}\|^{\frac{1}{2}}\|\hat{\u}\|^{\frac{1}{2}}_{\H^2}\|\nabla\psi_n\|\|\Delta\psi_n\| \; \leq \frac{1}{12}\|\Delta\psi_n\|^2 + \frac{3}{4C}\|\hat{\u}\|^2_{\H^2}\|\nabla\psi_n\|^2,\\
|K_{12}| & = \Big|\int_\Omega 
           F'''(\hat{\varphi}) \psi_n (\nabla\hat{\varphi}\cdot \nabla\mu_n) dx  + \int_\Omega F''(\hat{\varphi})\nabla\psi_n \cdot \nabla\mu_n dx \Big| \no\\ & \leq \frac{1}{2}\|\nabla\mu_n\|^2 + \frac{C}{4}\|\F'''(\hat{\varphi})\nabla\hat{\varphi}\|_{\mathrm{L}^\infty}\|\psi_n\|^2 + \frac{C}{4} \| F''(\hat{\varphi})\|^2_{\mathrm{L}^\infty}\|\nabla\psi_n\|^2,\\
|K_{13}| & \leq \frac{1}{12}\|\Delta\psi_n\|^2 + \frac{3}{4C}\|F''(\hat{\varphi})\|^2_{\mathrm{L}^\infty}\|\psi_n\|^2 \\
|K_{14}| & \leq \frac{\nu}{8}\|\nabla\overline{\w}_n\|^2 + C\|\hat{\varphi}\|^2_{\mathrm{L}^4}\|\nabla\psi_n\|^2\\
|K_{15}| & \leq \|\w_e\|_{\L^\infty}\|\nabla\hat{\varphi}\|\|\psi_n\| 
 \leq C \|\w_e\|_{\H^1}\|\w_e\|_{\H^2} + \|\nabla\hat{\varphi}\|^2\|\psi_n\|^2.\\
      \end{align*}
      Adding all the estimates of the right-hand side of \eqref{linearise approx}, we obtain 
      \begin{align}
      & \frac{1}{2}\frac{d}{dt}(\|\overline{\w}_n\|^2 + \|\psi_n\|^2_{\mathrm{H}^1}) +\frac{\nu}{2} \|\nabla\overline{\w}_n\|^2 + \frac{2}{3}\|\Delta\psi_n\|^2 + \frac{1}{2} \|\nabla\mu_n\|^2 \no \\ \, & \quad\leq  C(2\|\hat{\u}\|^2_{\H^1_{div}} + \frac{3}{4}\|\nabla\hat{\varphi}\|^2 + \frac{3}{2})(\|\w_e\|^2_{\H^1}+\|\w_e\|^2_{\H^2}) + C\big(\|\hat{\u}\|^2_{\H^2} + \|\Delta\hat{\varphi}\|^2_{\mathrm{L}^\infty} + \|\nabla\hat{\varphi}\|^2_{\mathrm{L}^\infty} \no\\ & \qquad + \|\hat{\varphi}\|^2_{\mathrm{L}^4} +\|\nabla\varphi\|^2 +\|\F'''(\hat{\varphi})\nabla\hat{\varphi}\|_{\mathrm{L}^\infty} + \| F''(\hat{\varphi})\|^2_{\mathrm{L}^\infty}\big) (\|\overline{\w}_n\|^2 + \|\psi_n\|^2_{\mathrm{H}^1}).
\end{align}
Observe that, with the help of Theorem \ref{strong solution} the map $t \rightarrow G(\hat{\u}(t), \hat{\varphi}(t))$ belongs to $\mathrm{L}^1(0, T)$ where\\
\begin{align*}
G(\hat{\u}(t), \hat{\varphi}(t)) := \|&\hat{\u}(t)\|^2_{\H^2} + \|\Delta\hat{\varphi}(t)\|^2_{\mathrm{L}^\infty} + \|\nabla\hat{\varphi}(t)\|^2_{\mathrm{L}^\infty} + \|\hat{\varphi}(t)\|^2_{\mathrm{L}^4} + \|\nabla\hat{\varphi}(t)\|^2 \\
&+\|\F'''(\hat{\varphi}(t))\nabla\hat{\varphi}(t)\|_{\mathrm{L}^\infty} + \| F''(\hat{\varphi}(t))\|^2_{\mathrm{L}^\infty}
\end{align*}
Thus owing to Gronwall's lemma  and the estimates in Theorem \ref{elliptic lifting} yields
\begin{align*}
    \|\overline{\w}_n\|_{\mathrm{L}^\infty(0, T, \G_{div}) \cap \mathrm{L}^2(0, T, \V_{div})} &  \leq C\|\eta\|_{\mathcal{U}},\\
   \text{and} \; \;  \|\psi_n\|_{\mathrm{L}^\infty(0, T, \mathrm{H}^1) \cap \mathrm{L}^2(0, T, \mathrm{H}^2)} & \leq C \|\eta\|_{\mathcal{U}}.
\end{align*}
By comparison in \eqref{var w}, \eqref{var psi} we can indeed get uniform bound on $\partial_t\overline{\w}_n$ and $\partial_t\psi_n$ as 
\begin{align*}
    \|\partial_t\overline{\w}_n\|_{\mathrm{L}^2(0, T; \V_{div}')} &\leq C\|\eta\|_{\mathcal{U}}\\
    \|\partial_t\psi_n\|_{\mathrm{L}^2(0, T; (\mathrm{H}^1(\Omega))')} & \leq C\|\eta\|_{\mathcal{U}}.
\end{align*}
Then we can get a pair $(\overline{\w}, \psi)$ such that
\begin{align*}
    \overline{\w} &\in \mathrm{L}^\infty(0, T; \G_{div}) \cap \mathrm{L}^2(0, T; \V_{div}) \cap \mathrm{H}^1(0, T; \V'_{div}),\\
    \psi&\in \mathrm{L}^\infty(0, T; \mathrm{H}^1(\Omega)) \cap \mathrm{L}^2(0, T; \mathrm{H}^2(\Omega)) \cap \mathrm{H}^1(0, T; (\mathrm{H}^1(\Omega))').
\end{align*}
which is the weak limit of $(\overline{\w}_n, \psi_n)$. Then we can pass to limit as $n \rightarrow \infty$ in \eqref{var psi} - \eqref{var psi} and also able to verify the pair $(\w, \psi)$ with $\w = \overline{\w} + \w_e$ satisfy the weak formulation of \eqref{linearise equ}.\\
 Since the system \eqref{linearise equ} is a linear system, the uniqueness of weak solutions follows easily.
\end{proof}

\subsection{Differentiability of Control to State Operator}
Let us define the affine space of control with imposing the compatibility condition: $$\mathcal{U}_c = \{\h(x, t)|
\, \h\in\mathcal{U} \text{ with } M\Tilde{\h}|_{t=0}=\u_0|_{\Gamma}\}.$$ Consider the control to state operator $ \mathcal{S} : \mathcal{U}_c \rightarrow \mathcal{V}$ with given initial data $(\u_0, \varphi_0) \in \H^1_{div} \times \mathrm{H}^2$. As $ \mathcal{V} \subseteq \mathcal{W}$,  we can consider $\mathcal{S} \text{ from } \mathcal{U}_c \text{ to the weaker space } \mathcal{W}$.
\begin{definition}
    We define $\mathcal{S}: \mathcal{U}_c \rightarrow \mathcal{W}$ is Fr\'echet differentiable in $\mathcal{U}_c$ if for any $\hat{\h} \in \mathcal{U}_c$, there exist a linear operator $\mathcal{S}'(\hat{\h}): \mathcal{U}_c-\{\hat{\h}\} \rightarrow \mathcal{W}$ such that
\begin{align}\label{diff criteria}
    \lim_{\|\eta\|_{\mathcal{U}}\rightarrow 0}\frac{\|\mathcal{S}(\hat{\h}+ \eta)-\mathcal{S}(\hat{\h})-\mathcal{S}'(\hat{\h})(\eta)\|_{\mathcal{W}}}{\|\eta\|_{\mathcal{U}}} = 0
\end{align}
for any arbitrary small perturbation $\eta \in \mathcal{U}_c-  \{\hat{\h}\} = \{\h-\hat{\h}| \,  \h\in\mathcal{U}_c\}$.
\end{definition}
In the next theorem, we prove the  Fr\'echet differentiability of the control to state operator $\mathcal{S}$.

\begin{theorem}\label{diffrentiability}
    Let $F \in C^4(\mathbb{R})$ and satisfies the assumption \eqref{prop of F}. Also, assume $(\u_0, \varphi_0)$ be a given initial data. Then the control to state operator $\mathcal{S}: \mathcal{U}_c \rightarrow \mathcal{W}$ is Fr\'echet differentiable. Moreover, for any $\hat{\h} \in \mathcal{U}_c$, its  
    Fr\'echet derivative $\mathcal{S}'(\hat{\h})$ is given by
    \begin{align*}
        \mathcal{S}'(\hat{\h})(\eta) = (\w, \psi), \, \forall \eta \in \mathcal{U}_c - \{\hat{\h}\},
    \end{align*}
    where $(\w, \psi)$ is the unique weak solution of the linearized system \eqref{linearise equ}  with control $\eta$, which is linearized around a strong solution $(\hat{\u}, \hat{\varphi})$ of the system \eqref{equ P} with control $\hat{\h}$.
\end{theorem}
\begin{proof}
Let $\hat{\h}$ be a given fixed control and $(\hat{\u}, \hat{\varphi}) = \mathrm{S}(\hat{\h})$ be the strong solution of \eqref{equ P} with control $M\Tilde{\hat{\h}}$. Let $(\overline{\u}, \overline{\varphi})$ be the strong solution of the system \eqref{equ P} with control $M\Tilde{\hat{\h}} + M\Tilde{\mathbf{\eta}}$. Let $\z = \overline{\u} - \hat{\u}, \, \xi = \overline{\varphi} - \hat{\varphi}$. Then $(\z, \xi)$ satisfies,
\begin{equation}\label{1st difference equation}
\left\{
\begin{aligned}
    &\z_t - \nu \Delta \z + \z\cdot\nabla\z + \hat{\u}\cdot\nabla\z + \z\cdot\nabla\hat{\u} + \nabla\pi_{\z} = \mu_{\xi}\nabla\hat{\varphi} + \mu_{\xi}\nabla\xi + \mu_{\hat{\varphi}}\nabla\xi, \, \text{ in } Q,\\
    &\xi_t + \z\cdot\nabla\xi + \hat{\u}\cdot\nabla\xi + \z\cdot\nabla\hat{\varphi} = \Delta\mu_{\xi}, \, \text{ in } Q,\\
    & \text{ div }\z = 0,  \, \text{ in } Q,\\
    &\mu_{\xi} = - \Delta\xi + F'(\overline{\varphi}) - F'(\hat{\varphi}),  \, \text{ in } Q,\\
    & \z = M\Tilde{\mathbf{\eta}}, \, \text{ on } \Sigma,\\
    &\frac{\partial\xi}{\partial\n} = 0 = \frac{\partial\mu_{\xi}}{\partial\n} \, \text{ on } \Sigma,\\
    & \z(0) = 0 , \, \xi(0) = 0, \, \text{ in } \Omega.
\end{aligned}   
\right.
\end{equation}
where $\pi_{\z} = \pi_{\overline{\u}} - \pi_{\hat{\u}}$. Now from Proposition \ref{cds}, continuous dependence of weak solutions, we have
\begin{align}
 \|\z\|^2_{\mathrm{L}^\infty(0, T; \mathbb{L}^2)\cap\mathrm{L}^2(0, T; \H^1_{div})} + \|\xi\|^2_{\mathrm{L}^\infty(0, T; \mathrm{H}^1)\cap\mathrm{L}^2(0, T; \mathrm{H}^2)} + \|\nabla\mu_{\xi}\|^2 \leq \tilde{C}  \|\mathbf{\eta}\|^2_{\mathcal{U}}.  
\end{align}   
Let us define $\y = \z - \w, \, \rho = \xi - \psi$ where $(\w, \psi)$ is the solution of linearized system \eqref{linearise equ} corresponding to non zero boundary for velocity as $M\Tilde{\eta}$. Then, $(\y, \rho)$ satisfies,
\begin{equation}\label{2nd difference equation}
\left\{
\begin{aligned}
   & \y_t - \nu \Delta \y + \z\cdot\nabla\z + \hat{\u}\cdot\nabla\y + \y\cdot\nabla\hat{\u} + \nabla\pi_{\y} = \mu_{\rho}\nabla\hat{\varphi} + \mu_{\xi}\nabla\xi + \mu_{\hat{\varphi}}\nabla\rho,\, \text{ in } Q,\\
    &\rho_t + \y\cdot\nabla\hat{\varphi} + \hat{\u}\cdot\nabla\rho + \z\cdot\nabla\xi  = \Delta\mu_{\rho},\, \text{ in } Q,\\
    &\mu_{\rho}  = -\Delta\rho + F'(\overline{\varphi}) - F'(\hat{\varphi}) - F''(\hat{\varphi})\psi,\, \text{ in } Q,\\
      & \text{ div }\y  = 0,\, \text{ in } Q,\\
   & \y  = 0,\, \text{ on } \Sigma,\\
   & \frac{\partial\rho}{\partial\n} = 0 = \frac{\partial\mu_{\rho}}{\partial\n},\, \text{ on } \Sigma,\\
   & \y(0) = 0, \, \rho(0)  = 0, \, \text{ in } \Omega.
\end{aligned}   
\right.
\end{equation}
We have from Theorem \ref{lin equ}
\begin{align*}
    \|(\w, \psi)\|_{\mathcal{W}} \leq C\|\eta\|_{\mathcal{U}}.
\end{align*}
We aim to show 
\begin{align*}
    \frac{\|(\y, \rho)\|_{\mathcal{W}}}{\|\eta\|_{\mathcal{U}}} \rightarrow 0 \text{ as } \|\eta\|_{\mathcal{U}} \rightarrow 0,
\end{align*}
which will directly imply \eqref{diff criteria}. For that,
we multiply the equation $\eqref{2nd difference equation}_1$ by $\y$ to yield
\begin{align}
    \frac{1}{2}\frac{d}{dt}\|\y\|^2 + \nu \|\nabla\y\|^2 & = -(\z \cdot \nabla\z, \y) - (\y \cdot \nabla \hat{\u}, \y) + (\mu_{\rho}\nabla\hat{\varphi}, \y) + (\mu_{\xi}\nabla\xi, \y) + (\mu_{\hat{\varphi}}\nabla\rho, \y) \label{test y} \no\\ & := \sum_{i=1}^{6} I_i 
\end{align}
Again we multiply the equation $\eqref{2nd difference equation}_2$ and $\eqref{2nd difference equation}_3$ by $\rho - \Delta \rho$ and $-\Delta \mu_{\rho} + \Delta \rho$  respectively and adding them together we get 
\begin{align}
    \frac{1}{2}\frac{d}{dt}\|\rho\|^2_{\mathrm{H}^1} + \|\nabla\mu_{\rho}\|^2 + \|\Delta\rho\|^2 & = -(\y\cdot\nabla\hat{\varphi}, \rho) - (\z\cdot\nabla\xi, \rho)
    + (\y\cdot\nabla\hat{\varphi}, \Delta\rho)   + (\hat{\u}\cdot\nabla\rho, \Delta\rho) + (\z\cdot\nabla\xi, \Delta\rho) \no \\ 
    &  +(F'(\overline{\varphi}) - F'(\hat{\varphi}) - F''(\hat{\varphi})\psi, \Delta\rho)  - ( F'(\overline{\varphi}) - F'(\hat{\varphi}) - F''(\hat{\varphi})\psi, \Delta\mu_{\rho}) \label{test rho}\no\\
    & := \sum_{i=1}^{7}J_i  
\end{align}
We estimate the right-hand side of \eqref{test y} by using H{\"o}lder inequality, Agmon's inequality, Young's inequality, and Sobolev inequality and derive the following,

\begin{align*}
    |I_1| & \leq C \|\z\|_{\L^4}\|\nabla\z\|\|\y\|_{\L^4} \; \leq C \|\z\|^4_{\H^1} + \frac{\nu}{8} \|\nabla\y\|^2\\
    |I_2| & \leq \|\y\|_{\L^4}\|\nabla\hat{\u}\|\|\y\|_{\L^4} \;  \leq C\|\nabla\hat{\u}\|^2\|\y\|^2 + \frac{\nu}{8}\|\nabla\y\|^2\\
    |I_3| & \leq \|\nabla\mu_{\rho}\|\|\hat{\varphi}\|_{\mathrm{L}^\infty}\|\y\| \;
     \leq \frac{1}{4}\|\nabla\mu_{\rho}\|^2 + \|\hat{\varphi}\|^2_{\mathrm{H}^2}\|\y\|^2\\
    |I_4| & = |(\xi\nabla\mu_{\xi}, \y)|\; 
    \leq \|\xi\|_{\mathrm{L}^4}\|\nabla\mu_{\xi}\|\|\y\|_{\L^4}\; \leq C \|\xi\|^2_{\mathrm{L}^4} \|\nabla\mu_{\xi}\|^2 + \frac{\nu}{8}\|\nabla\y\|^2\\
    |I_5| &\leq \|\nabla\mu_{\hat{\varphi}}\|\|\rho\|_{\mathrm{L}^4}\|\y\|_{\L^4}\; \leq C \|\mu_{\hat{\varphi}}\|^2\|\rho\|^2_{\mathrm{H}^1} + \frac{\nu}{8}\|\nabla\y\|^2.
    \end{align*}
    Substituting the estimates of $I_1 \text{ to } I_5$ in \eqref{test y},
    \begin{align}\label{equ y}
        \frac{1}{2}\frac{d}{dt}\|\y\|^2 + \frac{\nu}{2} \|\nabla\y\|^2  &\leq  C(\|\nabla\hat{\u}\|^2 + \|\nabla\hat{\varphi}\|^2)\|\y\|^2 + C \|\nabla\mu_{\hat{\varphi}}\|^2\|\rho\|^2_{\mathrm{H}^1} + \frac{1}{4}\|\nabla\mu_{\rho}\|^2  + C(\|\z\|^4_{\H^1} + \|\xi\|^2_{\mathrm{H}^1}\|\nabla\mu_{\xi}\|^2)\no\\
       & \leq C(\|\nabla\hat{\u}\|^2 + \|\nabla\hat{\varphi}\|^2)\|\y\|^2 + C \|\nabla\mu_{\hat{\varphi}}\|^2\|\rho\|^2_{\mathrm{H}^1} + \frac{1}{4}\|\nabla\mu_{\rho}\|^2 + C\|\eta\|^4_{\mathcal{U}}.
    \end{align}
   Now we will estimate the right-hand side of \eqref{test rho} as follows,
    \begin{align*}
    |J_1| & \leq \frac{1}{2}\|\y\|^2 + \frac{1}{2}\|\hat{\varphi}\|^2_{\mathrm{H}^3}\|\rho\|^2\\
    |J_2| & \leq \|\z\|_{\L^4}\|\nabla\xi\|\|\rho\|_{\mathrm{L}^4}\;
     \leq \frac{1}{2} \|\z\|^2_{\H^1}\|\nabla\xi\|^2 + \frac{1}{2} \|\rho\|^2_{\mathrm{H}^1}\\
    |J_3| & \leq \|\y\|\|\nabla\hat{\varphi}\|_{\mathrm{L}^\infty}\|\Delta\rho\|\; \leq C \|\hat{\varphi}\|^2_{\mathrm{H}^3}\|\y\|^2 + \frac{1}{8}\|\Delta\rho\|^2\\
    |J_4| & \leq C \|\hat{\u}\|^2_{\H^2}\|\nabla\rho\|^2 + \frac{1}{8}\|\Delta\rho\|^2\\
    |J_5| & \leq \|\z\|_{\L^4}\|\nabla\xi\|_{\mathrm{L}^4}\|\Delta\rho\|\; \leq C \|\z\|^2_{\H^1}\|\xi\|^2_{\mathrm{H}^2} + \frac{1}{8} \|\Delta\rho\|^2
\end{align*}
To estimate the sixth term of the right-hand side of \eqref{test rho}, from the Taylor series expansion of $F$ we note that,
\begin{align*}
    F'(\overline{\varphi}) - F'(\hat{\varphi}) - F''(\hat{\varphi})\psi = \frac{1}{2}F'''(\theta\overline{\varphi} + (1-\theta)\hat{\varphi})\xi^2 + F''(\hat{\varphi})\rho,
\end{align*}
for some $\theta \in (0, 1).$ Then, we can estimate the sixth term as 
\begin{align*}
    |J_6| & \leq C \|F'''(\theta\overline{\varphi} + (1-\theta)\hat{\varphi})\|_{\mathrm{L}^\infty}\|\xi^2\|\|\Delta\rho\| + \|F''(\hat{\varphi})\|_{\mathrm{L}^\infty}\|\rho\|\|\Delta\rho\|\no\\
    & \leq C \|\xi\|^4_{\mathrm{L}^4} + \frac{1}{8}\|\Delta\rho\|^2 + C\|\rho\|^2. 
\end{align*} 

The last term of \eqref{test rho} can be written as,
\begin{align}
  ( F'(\overline{\varphi}) - F'(\hat{\varphi}) - F''(\hat{\varphi})\psi, -\Delta\mu_{\rho}) & = \no\\ \Big(\frac{1}{2}F^{(4)}(\theta\overline{\varphi}+(1-\theta)\hat{\varphi})\nabla(\theta\overline{\varphi}+(1-\theta)\hat{\varphi})\xi^2 &+ F'''(\theta\overline{\varphi}+(1-\theta)\hat{\varphi})\xi\nabla\xi F'''(\hat{\varphi})\nabla\hat{\varphi}\rho + F''(\hat{\varphi})\nabla\rho, \nabla\mu_{\rho}\Big) \no\\
 &=: J_{7}^1 + J_{7}^{2} + J_{7}^{3} + J_{7}^{4},
\end{align}
where $J_{7}^1, \,  J_{7}^{2}, \, J_{7}^{3} , \, J_{7}^{4}$ can be estimated as follows
\begin{align*}
    |J_{7}^{1}| & \leq C \|F^{(4)}(\theta\overline{\varphi}+(1-\theta)\hat{\varphi})\|_{\mathrm{L}^\infty}(\|\nabla\hat{\varphi}\|_{\mathrm{L}^4} + \|\nabla\overline{\varphi}\|_{\mathrm{L}^4})\|\xi^2\|_{\mathrm{L}^4}\|\nabla\mu_{\rho}\|\no\\
    & \leq \frac{1}{8}\|\nabla\mu_{\rho}\|^2 + C(\|\nabla\hat{\varphi}\|^2_{\mathrm{L}^4} + \|\nabla\overline{\varphi}\|^2_{\mathrm{L}^4})\|\xi\|^4_{\mathrm{L}^8}\; \leq \frac{1}{8}\|\nabla\mu_{\rho}\|^2 + C(\|\nabla\hat{\varphi}\|^2_{\mathrm{L}^4} + \|\nabla\overline{\varphi}\|^2_{\mathrm{L}^4})\|\xi\|^4_{\mathrm{H}^1},\\
    |J_{7}^{2}| & \leq \| F'''(\theta\overline{\varphi}+(1-\theta)\hat{\varphi})\|_{\mathrm{L}^\infty}\|\xi\|_{\mathrm{L}^4}\|\nabla\xi\|_{\mathrm{L}^4}\|\nabla\mu_{\rho}\|\; \frac{1}{8}\|\nabla\mu_{\rho}\|^2 + C\|\xi\|^2_{\mathrm{L}^4}\|\nabla\xi\|^2_{\mathrm{L}^4},\\
    |J_{7}^{3}| & \leq \|F'''(\hat{\varphi})\|_{\mathrm{L}^\infty}\|\nabla\hat{\varphi}\|_{\mathrm{L}^4}\|\rho\|_{\mathrm{L}^4}\|\nabla\mu_{\rho}\|\; \leq C \|\nabla\hat{\varphi}\|^2_{\mathrm{L}^4}\|\rho\|^2_{\mathrm{H}^1} + \frac{1}{8}\|\nabla\mu_{\rho}\|^2,\\
    |J_{7}^{4}| & \leq \|F''(\hat{\varphi})\|_{\mathrm{L}^\infty}\|\nabla\rho\|\|\nabla\mu_{\rho}\| \; \leq C \|\nabla\rho\|^2 + \frac{1}{8}\|\nabla\mu_{\rho}\|^2.
\end{align*}
Substituting $J_1 \text{ to } J_7$ in \eqref{test rho} we get 
\begin{align}\label{equ rho}
    \frac{1}{2}\frac{d}{dt}\|\rho\|^2_{\mathrm{H}^1} + &\|\nabla\mu_{\rho}\|^2 +\frac{1}{2} \|\Delta\rho\|^2  \leq  \Big(\frac{1}{2} + C + C\|\hat{\varphi}\|^2_{\mathrm{H}^3} + \|\hat{\u}\|^2_{\H^2}\Big)\|\rho\|^2_{\mathrm{H}^1}  + C(\frac{1}{2} + \|\hat{\varphi}\|^2_{\mathrm{H}^3} )\|\y\|^2\no\\
    + &\frac{1}{2}\|\nabla\mu_{\rho}\|^2 +\Big(\frac{1}{2}\|\z\|^2_{\H^1}\|\nabla\xi\|^2 + C\|\z\|^2_{\H^1}\|\xi\|^2_{\mathrm{H}^2} + C\|\xi\|^4_{\mathrm{L}^4} + + C\|\xi\|^2_{\mathrm{L}^4}\|\nabla\xi\|^2_{\mathrm{L}^4}  \no\\
     & + C(\|\nabla\hat{\varphi}\|^2_{\mathrm{L}^4} + \|\nabla\overline{\varphi}\|^2_{\mathrm{L}^4})\|\xi\|^4_{\mathrm{H}^1}\Big)  \no\\
   \leq & \Big(\frac{1}{2} + C + C\|\hat{\varphi}\|^2_{\mathrm{H}^3} + \|\hat{\u}\|^2_{\H^2}\Big)\|\rho\|^2_{\mathrm{H}^1}  + C(\frac{1}{2} + \|\hat{\varphi}\|^2_{\mathrm{H}^3} )\|\y\|^2 + \frac{1}{2}\|\nabla\mu_{\rho}\|^2 \no\\ & +C(\|\nabla\hat{\varphi}\|^2_{\mathrm{L}^4} + \|\nabla\overline{\varphi}\|^2_{\mathrm{L}^4} + 3 + \frac{1}{2C})\|\eta\|^4_{\mathcal{U}}. 
\end{align}
Adding the inequalities \eqref{equ y} and \eqref{equ rho} we obtain
\begin{align}\label{resulting y rho}
   & \frac{1}{2}\frac{d}{dt}(\|\y\|^2 + \|\rho\|^2_{\mathrm{H}^1} )+ \frac{\nu}{2}\|\nabla\y\|^2 + \frac{1}{2}\|\Delta\rho\|^2 + \frac{1}{4}\|\nabla\mu_{\rho}\|^2 \no\\
 &\leq C\Big(\|\nabla\hat{\u}\|^2 + \|\nabla\hat{\varphi}\|^2 + 2\|\hat{\varphi}\|^2_{\mathrm{H}^3} + \|\nabla\mu_{\hat{\varphi}}\|^2
 + \|\nabla\hat{\u}\|^2_{\H^2} + 2\Big)(\|\y\|^2 + \|\rho\|^2_{\mathrm{H}^1}) \no\\
 &+ C\big(\|\nabla\hat{\varphi})\|^2_{\mathrm{L}^4} + \|\nabla\overline{\varphi}\|^2_{\mathrm{L}^4} +1 )\|\eta\|^4_{\mathcal{U}}.
\end{align}
As $(\hat{\u}, \hat{\varphi}) \text{ and } (\overline{\u}, \overline{\varphi})$ are strong solution of \eqref{equ P} with boundaries $M\Tilde{\hat{\h}} \text{ and } M\Tilde{\hat{\h}} + M\Tilde{\eta}$ respectively, so the pairs $(\hat{\u}, \hat{\varphi}) \text{ and } (\overline{\u}, \overline{\varphi})$ are uniformly bounded in $\mathcal{V}$ from Theorem \ref{strong solution}, i.e. 
\begin{align*}
    \|(\hat{\u}, \hat{\varphi})\|_{\mathcal{V}} &\leq M_1,\\
    \|(\overline{\u}, \overline{\varphi})\|_{\mathcal{V}} & \leq M_2.
\end{align*}
Therefore we can write \eqref{resulting y rho} as 
\begin{align}\label{y rho}
    \frac{1}{2}\frac{d}{dt}(\|\y\|^2 + \|\rho\|^2_{\mathrm{H}^1}) + \frac{\nu}{2})\|\nabla\y\|^2 + \frac{1}{2}\|\Delta\rho\|^2 + \frac{1}{4}\|\nabla\mu_{\rho}\|^2  \leq & C(\|\y\|^2 + \|\rho\|^2_{\mathrm{H}^1}) + C \|\eta\|^4_{\mathcal{U}},
\end{align} 
where $C$ depends on $\Tilde{C}, M_1, M_2, \Omega.$ Since the initial condition $(\y(0), \varphi(0)) = (0, 0)$, applying Gronwall lemma on \eqref{y rho} yields
\begin{align}
    \|(\y, \rho)\|^2_{\mathcal{W}} \leq C_T \|\eta\|^4_{\mathcal{U}},
\end{align}
Where $C_T$ depends on $M, M_1, M_2, \Omega, T$. Thus
\begin{align*}
    \lim_{\|\eta\|_{\mathcal{U}}\rightarrow 0}\frac{\|\mathcal{S}(\hat{\h}+\eta)-\mathcal{S}(\hat{\h})-\mathcal{S}'(\hat{\h})(\eta)\|_{\mathcal{W}}}{\|\eta\|_{\mathcal{U}}} = \lim_{\|\eta\|_{\mathcal{U}}\rightarrow 0}\frac{\|(\y, \rho)\|_{\mathcal{W}}}{\|\eta\|_{\mathcal{U}}} \leq C_T\|\eta\|_{\mathcal{U}} \rightarrow 0 \text{ as } \|\eta\|_{\mathcal{U}} \rightarrow 0.
\end{align*}
This completes the proof.
\end{proof}

\section{The First Order Necessary Optimality Condition}
Once we have shown that the control to state operator $\mathcal{S}$ is Fre\`chet differentiable, the next step is to derive the first-order necessary optimality condition. In this section, we establish the first-order necessary condition of optimality satisfied by the optimal solution and the solution of the linearized system \eqref{linearise equ}.

 \begin{theorem}\label{Optimality Condition by linear}
 Let the Assumption \ref{prop of F} be satisfied and $(\u_0, \varphi_0) \in \V_{div} \times \mathrm{H}^2$. Also suppose $(\u^\ast, \varphi^\ast, \h^\ast)$ be the  optimal triplet where $\h^\ast \in \mathcal{U}_{ad}$ and $S(\h^\ast) = (\u^\ast, \varphi^\ast).$  Let $(\w, \psi)$ be the unique weak solution of linearized problem \eqref{linearise equ} with  boundary data $M\Tilde{\h} - M\Tilde{\h^\ast}$,
 for any $\h \in \mathcal{U}$. Then following variational inequality holds:
 \begin{align}
     &\int_{Q}(\u^\ast - \u_{Q})\cdot \w dxdt + \int_{Q}(\varphi^\ast - \varphi_{Q})\cdot \psi dxdt + \int_{\Omega} (\u^\ast(T) - \u_{\Omega})\cdot \w(T) dx \no \\ & + \int_{\Omega} (\varphi^\ast(T) - \varphi_{\Omega})\cdot \psi(T) + \int_{\Sigma'} \h^\ast(\h - \h^\ast) dSdt \geq 0 \label{necessary condition}.
 \end{align}
 \end{theorem}
 
 \begin{proof}
 Since $\mathcal{U}_{ad}$ is a nonempty, convex subset of $\mathcal{U}$, then from Lemma 2.21 of \cite{fredi_book}, we have the reformulated cost functional $\Tilde{\mathcal{J}}$ satisfies 
 \begin{align}
     \tilde{\mathcal{J}}'_{\h}(\u^\ast, \varphi^\ast, \h^\ast)(\h - \h^\ast) \geq 0 \quad \forall \ \h \in \mathcal{U}_{ad},\label{1st nec cond}
 \end{align}
 where $\tilde{\mathcal{J}}'_{\h}(\u, \varphi, \h)$ denotes the Gateaux derivative of $\tilde{\mathcal{J}}$ with respect to $\h$.
 Now we will determine $\tilde{\mathcal{J}}'_{\h}(\u^\ast, \varphi^\ast, \h^\ast)$. We have from \eqref{reformulated} that $\tilde{\mathcal{J}}(\h) = \tilde{\mathcal{J}}(S(\h), \h)$, such that, $\mathcal{S}(\h) = (\u, \varphi)$ is the unique strong solution of \eqref{equ P} corresponding to control $\h$.
 Since $\tilde{\mathcal{J}}$ is a quadratic functional, by chain rule, we can write
 \begin{align}
\tilde{\mathcal{J}}'_{\h}(\u, \varphi, \h)= \tilde{\mathcal{J}}'_{\h}(S(\h), \h) = \tilde{\mathcal{J}}'_{S(\h)}(S(\h), \h) \circ S'(\h) + \tilde{\mathcal{J}}'_{\h}(S(\h), \h).\label{chain rule}
  \end{align}
In the above equation \eqref{chain rule}, we have for a fixed $\h \in \mathcal{U}$, the Gateaux derivative of $\tilde{\mathcal{J}}(S(\h), \h)$ with respect to $S(\h) = (\u,\varphi)$ and $\h$ are denoted by $\tilde{\mathcal{J}}'_{S(\h)}$ and $\tilde{\mathcal{J}}'_{\h}$, respectively. Now, we have the Gateaux derivative $\tilde{\mathcal{J}}'_{S(\h)}$ at $(\u^\ast, \varphi^\ast, \h^\ast)$ in the direction of  $(\y_1, y_2)$ is given by
   \begin{align}
       \tilde{\mathcal{J}}'_{S(h)}(S(\h^\ast), \h^\ast)(\y_1, y_2) = &\int_{Q}(\u^\ast - \u_{Q})\cdot\y_1 dx dt + \int_{Q}(\varphi^\ast - \varphi_{Q})y_2 dxdt \no \\ & + \int_{\Omega}(\u^\ast(T) - \u_{\Omega})\cdot \y_1(T) dx + \int_{\Omega}(\varphi(T) - \varphi_{\Omega}) y_2(T) dx,\label{1sr part chain}
   \end{align}
   for any $(\y_1, y_2) \in \mathcal{W}$.
   Similarly, we calculate the Gateaux derivative of $\tilde{\mathcal{J}}_{\h}$ at $(\u^\ast, \varphi^\ast, \h^\ast)$ in the direction of $\mathbf{g}$ as 
   \begin{align}
       \tilde{\mathcal{J}}'_{\h}(S(\h^\ast), \h^\ast)(\mathbf{g}) = \int_{\Sigma'}\h^\ast\cdot \mathbf{g} \, dSdt,\label{2nd part chain}
   \end{align}
   for any $\mathbf{g} \in \mathcal{U}$.
   Also from the Theorem \ref{diffrentiability}, we get that
   \begin{align}
       \mathcal{S}'(\h^\ast)(\h - \h^\ast) = (\w, \psi).\label{diff  c to s}
   \end{align}
   Now using \eqref{1sr part chain}, \eqref{2nd part chain}, \eqref{diff c to s} in \eqref{chain rule} we obtain,
   \begin{align*}
   \tilde{\mathcal{J}}'_{\h}(\u^\ast, \varphi^\ast, \h^\ast)(\h - \h^\ast) = &\int_{Q}(\u^\ast - \u_{Q})\cdot \w dxdt + \int_{Q}(\varphi^\ast - \varphi_{Q})\cdot \psi dxdt + \int_{\Omega} (\u^\ast(T) - \u_{\Omega})\cdot \w(T) dx \no \\ & + \int_{\Omega} (\varphi^\ast(T) - \varphi_{\Omega})\cdot \psi(T) \, dx + \int_{\Sigma'} \h^\ast\cdot(\h - \h^\ast) dSdt . 
   \end{align*}
   Therefore, we can conclude \eqref{necessary condition} from \eqref{1st nec cond}.
 \end{proof}
 \subsection{First Order Necessary Optimality Condition Via Adjoint System}
 
 In this section, we would like to simplify the optimality condition \eqref{necessary condition} and write it in terms of the optimal solution and adjoint variables. This optimality system can serve as the basis for computing approximations to optimal solutions numerically. Thus we will now derive the adjoint system corresponding to the system \eqref{equ P} by using the Lagrange multipliers method. We know that the adjoint system variables act as Lagrange multipliers corresponding to the state system variables \eqref{equ P}.
 
 It is well-known that the necessary optimality conditions satisfied by the optimal control can be derived from the exact Lagrange multipliers method, which is also known as Karush-Kuhn-Tucker (KKT) theory for optimization problems in Banach spaces. This method is described for various elliptic and parabolic problems in \cite{fredi_book, fursikov_book}. The application of the method for nonlinear PDEs is difficult because it requires a lot of experience in matching the operators, functionals, and spaces involved. Despite the technical difficulty, KKT theory has been applied to study optimal boundary control problems for Navier-Stokes equations in \cite{Fursikov_bndry, Fursikov_3dbcontl, Fursikov_3dbdary}, but it is not straight forward to use in our case for CHNS system  \eqref{equ P} since our system is highly nonlinear coupled pde system. Thus we will use a formal Lagrange multiplier method [see section 2.10 in \cite{fredi_book}] to derive the adjoint system, then prove the well-posedness of the adjoint system, and finally establish the necessary condition of optimality for our optimal control problem $(\mathbf{OCP})$. 
 
 For this purpose, we formally introduce the Lagrange functional for the control problem $(\mathbf{OCP})$ as follows:
 \begin{align}\label{lf}
     \mathcal{L}(( \u, \varphi), &\h, (\p, \zeta, \hat{P}, \p_1, \zeta_1)) \no\\
     :=  \mathcal{J}&(\u, \varphi, \h) -\int_{Q} [\mathbf{u}_t - \nu \Delta\mathbf{u} + (\mathbf{u}\cdot \nabla)\mathbf{u} + \nabla \pi - (-\Delta\varphi + F'(\varphi))\nabla \varphi]\cdot\p \, dxdt \no\\
     & -\int_{Q}[\varphi_t + \mathbf{u} \cdot \nabla \varphi- \Delta(-\Delta\varphi + F'(\varphi))]\cdot\zeta \, dxdt - \int_{Q} ( div~\mathbf{u})\hat{P} \, dxdt \no\\
     & - \int_{\Sigma} (\u-M\tilde\h)\cdot\p_1 \, dSdt - \int_{\Sigma} \frac{\partial\varphi}{\partial\n}\zeta_1 \,  dSdt 
 \end{align}
     for any $\h \in \mathcal{U}_{ad} \text{ and } (\u, \varphi) = \mathcal{S}(\h).$ Here $\zeta, \p, \hat{P}, \zeta_1, \p_1$ are Lagrange multipliers corresponding to five state constraints in \eqref{equ P} respectively.\\
     Let $(\u^\ast, \varphi^\ast, \h^\ast)$ be the optimal solution to the problem $(\mathbf{OCP})$ such that $(\u^\ast, \varphi^\ast) = \mathcal{S}(\h^\ast)$. Then by using the Lagrange principle, we conclude that $(\u^\ast, \varphi^\ast, \h^\ast)$ together with Lagrange multipliers $\p, \zeta, \hat{P}, \p_1, \zeta_1$, satisfies the first order optimality condition associated with the optimization problem related to the Lagrange functional $\mathcal{L}$, defined as follows:
     \begin{align*}
         \min _{\h \in \mathcal{U}_{ad}} \mathcal{L}((\u, \varphi), &\h, (\p, \zeta, \hat{P}, \p_1,\zeta_1)). 
     \end{align*}
     Now, since $(\u, \varphi)$ has become formally unconstrained, the Fr\'echet derivative of $\mathcal{L}$ with respect to $(\u, \varphi)$ will vanish at the optimal point $(\u^\ast, \varphi^\ast, \h^\ast)$, which implies
     \begin{align}\label{derivative lf}
          \mathcal{L}'_{(\u, \varphi)}((\u^\ast, \varphi^\ast), \h^\ast, (\p, \zeta, \hat{P}, \p_1, \zeta_1))(\u_1, u_2) = 0
     \end{align}
  for all smooth function $(\u_1, u_2)$ such that 
  \begin{align*}
      \u_1(0, .) = 0, \, \, u_2(0, .) = 0 \, \, \text { in } \Omega
  \end{align*}
Furthermore, from the Lagrange principle, the constraints on $\h$ in \eqref{ad control space} gives the following variational inequality
  \begin{align}\label{derivative lh}
          \mathcal{L}'_{\h}((\u^\ast, \varphi^\ast), \h^\ast, (\p, \zeta, \hat{P}, \p_1, \zeta_1))(\h-\h^\ast) \geq 0
     \end{align}
     for all $\h \in \mathcal{U}_{ad}$. \\
  Next, we want to determine the Lagrange multipliers $\p, \zeta, \hat{P}, \p_1, \zeta_1$ from \eqref{derivative lf}, which are the adjoint states corresponding to state equations \eqref{equ P}. For this purpose, we formally calculate \eqref{derivative lf}, then perform integration by parts and take the terms together corresponding to $(\u_1, u_2, \pi)$ to derive the following linear system satisfied by $(\p, \zeta, \hat{P})$:
\begin{equation}\label{adjoint equation}
\left\{
\begin{aligned}
     &-\partial_t\p - \nu \Delta\p + (\u^\ast \cdot \nabla)\p - (\p \cdot \nabla^{T})\u^\ast + \zeta\nabla\varphi^\ast - \nabla\hat{P} = \u^\ast - \u_{Q},  \quad \text{ in } Q,\\
     &-\partial_t\zeta - \u^\ast \cdot \nabla\zeta - \p \cdot \nabla(\Delta\varphi^\ast) + 
     \text{div}((\nabla\p) \cdot\nabla\varphi^\ast) + \text{div}((\nabla^T(\nabla\varphi^\ast))\cdot\p ) \\
     &+ \Delta^2\zeta - F''(\varphi^\ast)\Delta\zeta  = \varphi^\ast - \varphi_{Q}, \quad \text{ in } Q, \\
    & \text{div}~\p = 0,  \quad \text{ in } Q,\\  
    & \p = 0, \quad \text{ on } \Sigma,\\
    & \frac{\partial\zeta}{\partial\n} = 0 = \frac{\partial(\Delta\zeta)}{\partial\n}, \quad \text{ on } \Sigma,\\
    & \p(T) = \u^\ast(T)-\u_{\Omega}, \, \zeta(T) = \varphi^\ast(T) - \varphi_{\Omega}, \quad \text{ in } \Omega.
\end{aligned}   
\right.
\end{equation}
 Furthermore, on the boundary $\Sigma$, the two lagrange multipliers $\p_1, \ \zeta_1$ can be uniquely determined in terms of $(\p, \zeta, \hat{P})$ and  it satisfy the following equations:
 \begin{align}
     & \p_1 + \hat{P}\n + \frac{\partial\p}{\partial\n} = 0,  \quad \text{ on } \Sigma\label{bdry p}\\
     &\frac{\partial\zeta_1}{\partial\n} + [(\nabla\p) \cdot\nabla\varphi^\ast + (\nabla^T(\nabla\varphi^\ast))\cdot\p ]\n = 0,  \quad \text{ on } \Sigma. \label{bdry zeta}
 \end{align}
We call the linear pde system \eqref{adjoint equation} as the adjoint system corresponding to \eqref{equ P}.
\begin{remark}
   The expression of Lagrange functional $\mathcal{L}$ in \eqref{lf} is not well-defined yet because we only have the regularity of $(\u, \varphi) \in \mathcal{V}$ for the control $\h \in \mathcal{U}$. We do not know the regularity of Lagrange multipliers $\zeta, \p, \hat{P}, \zeta_1, \p_1$ yet. Thus the Lagrange multiplier method presented in this section has the sole purpose of identifying the correct form of the adjoint system. 
  \end{remark}
  Now we establish the following existence result of the adjoint system \eqref{adjoint equation}.
 \begin{theorem}\label{existence adj equ}
    Let $(\u^\ast, \varphi^\ast) \in \mathcal{V}$ and assumption \eqref{prop of F} on $F$ be satisfied. Also, assume $\u_\Omega \in \L^2_{div}(\Omega), \, \varphi_\Omega \in \mathrm{H}^1(\Omega)$. Then the linear problem \eqref{adjoint equation} has a unique solution $(\p, \zeta)$ such that
    \begin{align*}
        \p \in & \mathrm{L}^\infty(0, T; \G_{div})\cap\mathrm{L}^2(0, T; \V_{div})\cap\mathrm{H}^1(0, T, \V'_{div}),\\
        \zeta \in & \mathrm{L}^\infty(0, T; \mathrm{H}^1)\cap\mathrm{L}^2(0, T; \mathrm{H}^2\cap\mathrm{H}^3)\cap\mathrm{H}^1(0, T; (\mathrm{H}^1)').
    \end{align*}
 \end{theorem}
 \begin{proof}
The proof follows from a similar argument as in Theorem \ref{lin equ} using the Faedo - Galerkin approximation method. For the sake of simplicity, we omit the approximation scheme and do the apriori estimates only.\\
We multiply the equation $\eqref{adjoint equation}_1$ by $\p$ and the equation $\eqref{adjoint equation}_2$ by $\zeta - \Delta\zeta$ we get
 \begin{align}
     -\frac{1}{2}\frac{d}{dt}\|\p\|^2 + \nu \|\nabla\p\|^2 &= ( (\p\cdot\nabla^T)\u^\ast, \p ) + ( \zeta \nabla\varphi^\ast, \p ) + (\u^\ast-\u_{Q}, \p) \no\\
     & := \sum_{k=1}^{3}I_k \label{equ p}
     \end{align}
     \begin{align}
     -\frac{1}{2}\frac{d}{dt}\|\zeta\|^2_{\mathrm{H}^1} +& \|\Delta\zeta\|^2 + \|\nabla(\Delta\zeta)\|^2  = (\p \cdot \nabla(\Delta\varphi^\ast), \zeta) - (\text{div}((\nabla\p) \cdot\nabla\varphi^\ast), \zeta) \no\\
     & - (\text{div}((\nabla^T(\nabla\varphi^\ast))\cdot\p ), \zeta)+ (F''(\varphi^\ast)\Delta\zeta, \zeta) \; + (\varphi^\ast - \varphi_{Q}, \zeta)  - (\u^\ast \cdot \nabla\zeta, \Delta\zeta) + \no\\ 
     & +  (\p \cdot \nabla(\Delta\varphi^\ast), \Delta\zeta)  (\text{div}((\nabla\p) \cdot\nabla\varphi^\ast), \Delta\zeta) + (\text{div}((\nabla^T(\nabla\varphi^\ast))\cdot\p ), \Delta\zeta) \no \\
     &- (F''(\varphi^\ast)\Delta\zeta, \Delta\zeta) - (\varphi^\ast - \varphi_{Q}, \Delta\zeta) \no\\
     & := \sum_{i=1}^{11}J_i\label{equ zeta}
 \end{align}
 Now we estimate each $I_i$ and $J_i$ one by one using Poincare, Young's, and Sobolev inequality. We estimate $I_1$ using Ladyzhanskya and Young's inequality as 
 \begin{align*}
  |I_1| & \leq \|\p\|^2_{\L^4}\|\nabla\u^\ast\|\;
   \leq \sqrt{2}\|\p\|\|\nabla\p\|\|\nabla\u^\ast\| \;
   \leq \frac{\nu}{12}\|\nabla\p\|^2 + \frac{6}{\nu}\|\p\|^2\|\nabla\u^\ast\|^2.
 \end{align*}
 Similarly, we can estimate $I_2$ as 
 \begin{align*}
    |I_2| & \leq \|\nabla\zeta\|\|\varphi^\ast\|_{\mathrm{L}^4}\|\p\|_{\L^4}\;
     \leq \frac{C}{4}\|\nabla\zeta\|^2 + \frac{1}{4C}\|\varphi^\ast\|^2_{\mathrm{L}^4}\|\p\|\|\nabla\p\| \; \leq \frac{C}{4}\|\nabla\zeta\|^2 + \frac{\nu}{12}\|\nabla\p\|^2 + \frac{3}{16C\nu}\|\varphi^\ast\|^4_{\mathrm{L}^4}\|\p\|^2, \\
     |I_3| & \leq \frac{3}{4\nu}\|\u^\ast - \u_{Q}\|^2 + \frac{\nu}{12}\|\nabla\p\|^2
 \end{align*}
 Where $C$ is a generic constant. Adding $I_1 \text{ to } I_3$ in the equation \eqref{equ p} we get
 \begin{align}\label{total p}
    -\frac{1}{2}\frac{d}{dt}\|\p\|^2 + \frac{3\nu}{4} \|\nabla\p\|^2 &\leq \frac{3}{4\nu}\|\u^\ast-\u_Q\|^2 + \frac{C}{4}\|\nabla\zeta\|^2 + \Big(\frac{6}{\nu}\|\nabla\u^\ast\|^2 + \frac{3}{16C\nu}\|\varphi^\ast\|^4_{\mathrm{L}^4}\Big)\|\p\|^2.
 \end{align}
 Now we will estimate $J_i$.
 \begin{align*}
     |\J_1| & \leq \|\p\|\|\nabla\zeta\|\|\Delta\varphi^\ast\|_{\mathrm{L}^\infty} \; \leq \frac{1}{4C}\|\varphi^\ast\|^2_{\mathrm{H}^2}\|\p\|^2 + \frac{C}{4}\|\nabla\zeta\|,
 \end{align*}
 where we have used Agmon's inequality and Young's inequality. Similarly, $J_2$ can be estimated as 
\begin{align*}
    |J_2| & = |\big((\nabla\p) \cdot\nabla\varphi^\ast, \nabla\zeta\big)| \;
      \leq \frac{\nu}{12}\|\nabla\p\|^2 + \frac{6}{\nu}\|\varphi^\ast\|^2_{\mathrm{H}^3}\|\nabla\zeta\|^2.
\end{align*}
Similarly, 
\begin{align*}
    |J_3|  &= |\big((\nabla(\nabla\varphi^\ast))\cdot\p, \nabla\zeta\big)|  
\leq\|\p\|\|\nabla\zeta\|\|\nabla(\nabla\varphi^\ast)\|_{\mathrm{L}^\infty}\; \leq \|\varphi^\ast\|^2_{\mathrm{H}^4}\|\p\|^2 + \frac{1}{4}\|\nabla\zeta\|^2,\\
    |J_4| & \leq \|F''(\varphi^\ast)\|_{\mathrm{L}^\infty}\|\Delta\zeta\|\|\zeta\| 
           \leq C\|F''(\varphi^\ast)\|^2_{\mathrm{L}^\infty}\|\zeta\|^2 + \frac{1}{8}\|\Delta\zeta\|^2,\\
    |J_5| & \leq \frac{1}{2}\|\varphi^\ast - \varphi_{Q}\|^2 + \frac{1}{2}\|\zeta\|^2,\\
    |J_6| & \leq C\|\u^\ast\|^2_{\H^2}\|\nabla\zeta\|^2 + \frac{1}{8}\|\Delta\zeta\|^2,\\
    |J_7| & = |(\p \cdot \nabla(\Delta\zeta), \Delta\varphi^\ast) | \leq \frac{1}{12}\|\nabla(\Delta\zeta)\|^2 + \frac{3}{4}\|\Delta\varphi^\ast\|^2_{\mathrm{L}^\infty}\|\p\|^2,\\
    |J_8| &=  |(\text{div}( (\nabla\p) \cdot \nabla\varphi^\ast ), \Delta\zeta \, )| \leq \|\p\|_{\L^4}\|\nabla(\Delta\varphi^\ast)\|_{\mathrm{L}^4}\|\nabla(\Delta\zeta)\| \; \|\p\|^\frac{1}{2}\|\nabla\p\|^\frac{1}{2}\|\varphi^\ast\|^2_{\mathrm{H}^4}\|\nabla(\Delta\zeta)\|\\
          & \leq C\|\varphi^\ast\|^2_{\mathrm{H}^4}\|\p\|^2 + \frac{\nu}{12}\|\nabla\p\|^2 + \frac{1}{12}\|\nabla(\Delta\zeta)\|^2, \\
   |J_9| & =  |(\, (\nabla^T(\nabla\varphi^\ast))\cdot\p, \nabla(\Delta\zeta)\, )|\leq \frac{3}{4}\|\varphi^\ast\|^2_{\mathrm{H}^4} \|\p\|^2+ \frac{1}{12}\|\nabla(\Delta\zeta)\|^2,\\
    |J_{10}|  & \leq (C +1)\|\varphi^\ast\|^2_{\mathrm{H}^3}\|\nabla\zeta\|^2 + \frac{1}{5}\|\Delta\zeta\|^2 + \frac{1}{12}\|\nabla(\Delta\zeta)\|^2,\\
    |\J_{11}| & \leq \frac{1}{2}\|\varphi^\ast-\varphi_{Q}\|^2 + \frac{1}{8}\|\Delta\zeta\|^2.
\end{align*}
Taking estimates of $J_1 \text{ to } J_{4}$ into account, we write \eqref{equ zeta} as 
\begin{align}\label{total zeta}
    -\frac{1}{2}\frac{d}{dt}\|\zeta\|^2_{\mathrm{H}^1} + \frac{1}{2} \|\Delta\zeta\|^2 +\frac{2}{3}\|\nabla(\Delta\zeta)\|^2  \leq  \Big( \frac{1}{4C}\|\varphi^\ast\|^2_{\mathrm{H}^2} +(\frac{1}{4C} + \frac{5}{2} + C)\|\varphi^\ast\|^2_{\mathrm{H}^4} + 2\|\varphi^\ast\|^2_{\mathrm{H}^3}\Big)\|\p\|^2 \no\\ + \Big(\frac{6}{\nu} + \frac{C+1}{4} + C\|\F''(\varphi^\ast)\|^2_{\mathrm{L}^\infty} + C\|\u^\ast\|^2_{\H^2}\Big)\|\zeta\|^2_{\mathrm{H}^1}  + \frac{3}{4\nu}\|\u^\ast-\u_Q\|^2 +\frac{3}{8}\|\varphi^\ast - \varphi_{Q}\|^2 
\end{align}
Adding the inequalities \eqref{total p} and \eqref{total zeta} we get 
\begin{align}\label{sum p zeta}
    -\frac{1}{2}\frac{d}{dt}(\|\p\|^2 &+ \|\zeta\|^2_{\mathrm{H}^1}) + \frac{3\nu}{4} \|\nabla\p\|^2 + \frac{1}{2}  \|\Delta\zeta\|^2 +\frac{2}{3}\|\nabla(\Delta\zeta)\|^2 \leq \frac{3}{4\nu}\|\u^\ast-\u_Q\|^2 + \frac{1}{2}\|\varphi^\ast - \varphi_{Q}\|^2 \no\\
    & + \Big[\frac{6}{\nu}\|\nabla\u^\ast\|^2 +  \frac{3}{16C\nu}\|\varphi^\ast\|^4_{\mathrm{L}^4} + \frac{1}{4C}\|\varphi^\ast\|^2_{\mathrm{H}^2} +C\|\varphi^\ast\|^2_{\mathrm{H}^4} + 2\|\varphi^\ast\|^2_{\mathrm{H}^3}\Big]\|\p\|^2 \no\\&+  C\Big[1 + \|\F''(\varphi^\ast)\|^2_{\mathrm{L}^\infty} + \|\u^\ast\|^2_{\H^2}\Big]\|\zeta\|^2_{\mathrm{H}^1}.
\end{align}
Integrating the inequality \eqref{sum p zeta} from $t$ to $T$ yields
\begin{align}\label{int p zeta}
    &\frac{1}{2}(\|\p(t)\|^2 + \|\zeta(t)\|^2_{\mathrm{H}^1}) + \frac{3\nu}{4} \int_t^T\|\nabla\p(s)\|^2 ds + (\frac{1}{2}-C_3) \int_t^T\|\Delta\zeta(s)\|^2 ds +\frac{3}{4}\int_t^T\|\nabla(\Delta\zeta(s))\|^2 ds\no\\
    &\leq \big[\|\p(T)\|^2 + \|\zeta(T)\|^2 \big] + \frac{3}{4\nu}\int_t^T\|\u^\ast(s)-\u_Q\|^2 ds + \frac{1}{2}\int_t^T\|\varphi^\ast(s) - \varphi_{Q}\|^2 ds \no\\&
     +\int_t^T \Big[\frac{6}{\nu}\|\nabla\u^\ast(s)\|^2 +  \frac{3}{16C\nu}\|\varphi^\ast(s)\|^4_{\mathrm{L}^4} + \frac{1}{4C}\|\varphi^\ast(s)\|^2_{\mathrm{H}^2} +C\|\varphi^\ast(s)\|^2_{\mathrm{H}^4} + 2\|\varphi^\ast(s)\|^2_{\mathrm{H}^3}\Big]\|\p(s)\|^2 ds\no\\
    & + C\int_t^T \Big[1 + \|\F''(\varphi^\ast(s))\|^2_{\mathrm{L}^\infty} + \|\u^\ast(s)\|^2_{\H^2}\Big]\|\zeta(s)\|^2_{\mathrm{H}^1} ds,
\end{align}
 for all $t \in [0, T]$. Applying Gronwall's lemma to the inequality \eqref{int p zeta}
\begin{align}\label{gronwall}
    \frac{1}{2}(\|\p(t)\|^2 + \|&\zeta(t)\|^2_{\mathrm{H}^1}) \no\\
    \leq \Big[\|\p(T)\|^2& + \|\zeta(T)\|^2_{\mathrm{H}^1} + \frac{3}{4\nu}\int_0^T\|\u^\ast(s)-\u_Q\|^2 ds + \frac{1}{2}\int_0^T\|\varphi^\ast(s) - \varphi_{Q}\|^2 ds \Big]\no\\ \times exp\Big[T + &\int_0^T \|\F''(\varphi^\ast(s))\|^2_{\mathrm{L}^\infty} ds + \int_0^T\|\u^\ast(s)\|^2_{\H^2} ds\Big] \no\\
    \times exp\Big[\int_0^T&\frac{6}{\nu}\|\nabla\u^\ast(s)\|^2 ds+  \frac{3}{16C\nu}\int_0^T\|\varphi^\ast(s)\|^4_{\mathrm{L}^4} ds+ \frac{1}{4C}\int_0^T\|\varphi^\ast(s)\|^2_{\mathrm{H}^2}ds \no\\&+C\int_0^T\|\varphi^\ast(s)\|^2_{\mathrm{H}^4} ds+ 2\int_0^T\|\varphi^\ast(s)\|^2_{\mathrm{H}^3} ds\Big].
\end{align}
Since $(\u^\ast, \varphi^\ast)$ is the strong solution of nonlinear system \eqref{equ P}, the right-hand side of \eqref{gronwall} is finite. Therefore using \eqref{int p zeta} and \eqref{gronwall} we conclude that $\p \in \mathrm{L}^\infty(0, T; \G_{div}) \cap \mathrm{L}^2(0, T; \V_{div})$ and $\zeta \in \mathrm{L}^\infty(0, T; \mathrm{H}^1) \cap \mathrm{L}^2(0, T; \mathrm{H}^2\cap \mathrm{H}^3).$ Now with this regularity of $\p$, $\zeta$ and using the equations $\eqref{adjoint equation}_1$ and $\eqref{adjoint equation}_2$ we get a uniform estimates on time derivatives $\p_t, \, \zeta_t$ so that $\p_t \in \mathrm{L}^2(0, T; \V'_{div}), \, \zeta_t \in \mathrm{L}^2(0, T; (\mathrm{H}^1)').$ Hence we get a weak solution $(\p, \zeta)$ as claimed. The uniqueness of weak solutions follows from the linearity of the system. Since $(\p, \zeta)$ is the unique solution of the adjoint system, as in the Navier - Stokes equations (see \cite{Temam}), we can determine $\hat{P} \in \mathrm{L}^2(0, T; \mathrm{L}^2_{0}(\Omega))$, where $\mathrm{L}^2_{0}(\Omega) = \{ g \in \mathrm{L}^2(\Omega):\int_{\Omega} g(x) dx = 0\}$.
\end{proof}

\begin{remark}\label{reg adj}
Note that, if $\u_\Omega \in \H^1_{div}$ then solution of the adjoint system, $\p$ satisfy $\p \in \mathrm{L}^\infty(0, T; \V_{div}) \cap \mathrm{L}^2(0, T; \H^2)$. This can be shown by multiplying the equation $\eqref{adjoint equation}_1$ by $\mathbf{A}\p$ and proceeding similarly as in Theorem \ref{existence adj equ}.       
\end{remark}
\begin{corollary}\label{exist bdry multiplier}
Let $\u_\Omega\in \H^1_{div}$ and all the assumptions of Theorem \ref{existence adj equ} hold. Then the Lagrange multipliers $(\p_1, \zeta_1)$ are uniquely determined by the equations \eqref{bdry p} and \eqref{bdry zeta} such that
\begin{align}\label{reg p zeta}
    \p_1 \in  \mathrm{L}^2(0, T; \H^{\frac{1}{2}}(\Gamma)),
    \zeta_1 \in  \mathrm{L}^2(0, T; \mathrm{H}^{\frac{1}{2}}(\Gamma)).
\end{align}
\end{corollary}
\begin{proof}
  Since $(\p,\zeta,\hat{P})$ is the unique solution of the adjoint system \eqref{adjoint equation}, we can determine $(\p_1, \zeta_1)$ from the equations \eqref{bdry p} and \eqref{bdry zeta} uniquely on the boundary $\Sigma$. We only need to show \eqref{reg p zeta}. From \eqref{reg adj} and the trace theorem we have $\frac{\partial\p}{\partial\n} \in \mathrm{L}^2(0, T; \H^{\frac{1}{2}}(\Gamma))$ and $\hat{P}\n \in \mathrm{L}^2(0, T; \mathrm{H}^{\frac{1}{2}}(\Gamma))$. Therefore from \eqref{bdry p}, we have $\p_1 \in \mathrm{L}^2(0, T; \H^{\frac{1}{2}}(\Gamma))$. Moreover, using Gagliardo - Nirenberg interpolation inequality we obtain
  \begin{align*}
      \|\p\cdot\nabla\varphi^\ast\|^2_{\mathrm{H}^1} = \Big(\sum_{|\alpha| \leq 1}\int_\Omega|D^\alpha(\p\cdot\nabla\varphi^\ast)|^2\Big)\;  &\leq \sum_{|\alpha| \leq 1}\Big( \int_\Omega |D^\alpha\p \cdot \nabla\varphi^\ast|^2 + \int_\Omega|\p \cdot D^\alpha(\nabla\varphi^\ast)|^2 \Big)\\
      &\leq \|\p\|^2_{\W^{1,4}}\|\nabla\varphi^\ast\|^2_{\mathrm{L}^4} + \|\p\|^2_{\mathrm{L}^4}\|\varphi^\ast\|^2_{\W^{2,4}}\\
      & \leq C \|\p\|^{\frac{3}{2}}_{\H^2}\|\p\|^\frac{1}{2}_{\L^2} \|\nabla\varphi^\ast\|_{\mathrm{L}^4} + C\|\nabla\p\|^2\|\varphi^\ast\|^{\frac{5}{3}}_{\mathrm{H}^3}\|\varphi^\ast\|^\frac{1}{3}_{\mathrm{H}^1}.
  \end{align*}
  As a consequence, using the fact $(\u^\ast, \varphi^\ast)\in \mathcal{V}$ we obtain
  \begin{align*}
      \p\cdot\nabla\varphi^\ast \in \mathrm{L}^2(0, T; \mathrm{H}^1(\Omega)).
  \end{align*}
 Therefore, $(\nabla\p) \cdot\nabla\varphi^\ast + (\nabla^T(\nabla\varphi^\ast))\cdot\p = \nabla(\p\cdot\nabla\varphi^\ast) \in \mathrm{L}^2(0, T; \mathrm{L}^2(\Omega)).$ Which implies 
 \begin{align*}
   [(\nabla\p) \cdot\nabla\varphi^\ast + (\nabla^T(\nabla\varphi^\ast))\cdot\p]\n \in \mathrm{L}^2(0, T; \mathrm{H}^{-\frac{1}{2}}(\Gamma)).  
 \end{align*}
 Therefore, it follows from \eqref{bdry zeta} that
 \begin{align*}
     \frac{\partial\zeta_1}{\partial\n} \in \mathrm{L}^2(0, T; \mathrm{H}^{-\frac{1}{2}}(\Gamma)),
 \end{align*}
 and hence from Trace theorem,
 \begin{align*}
     \zeta_1 \in \mathrm{L}^2(0, T; \mathrm{H}^\frac{1}{2}(\Gamma)).
 \end{align*}
\end{proof}

\begin{remark}
    Using the Theorem \ref{existence adj equ} and Corollary \ref{exist bdry multiplier}, we can now conclude that the Lagrange functional $\mathcal{L}$ defined in \eqref{lf} is well-defined.
\end{remark}
In the following theorem, we will derive the first-order necessary optimality condition in terms of optimal solution of $(\mathbf{OCP})$ and the adjoint variables $\p, \zeta$.

\begin{theorem}\label{1st order necessary optimality condition}
 Let the Assumption \ref{prop of F} holds and  $(\u_0, \varphi_0) \in \V_{div} \times \mathrm{H}^2$. In addition, let $(\u^\ast, \varphi^\ast,\h^\ast )$ be the optimal triplet, i.e., $\h^\ast \in \mathcal{U}_{ad}$ be the optimal boundary control of the problem ($\mathbf{OCP}$) and $(\u^\ast, \varphi^\ast)$ be the strong solution of \eqref{equ P} corresponding to $\h^\ast$. Let $(\p, \zeta)$ be the solution of the adjoint system \eqref{adjoint equation}.  Then for any $\h \in \mathcal{U}_{ad}$, the following variational inequality holds:
 \begin{align}\label{variational ineqa}
      \int_{\Sigma'} \h^\ast(\h - \h^\ast) \, dSdt - \int_{\Sigma} M\hat{P}\n\cdot(\h - \h^\ast) \, dSdt - \int_{\Sigma} M\frac{\partial\p}{\partial\n}\cdot(\h - \h^\ast) \, dSdt \geq 0.
 \end{align}
\end{theorem}
\begin{proof}
We have from \eqref{derivative lh} that 
\begin{align*}
          \mathcal{L}'_{\h}((\u^\ast, \varphi^\ast), \h^\ast, (\p, \zeta, \hat{P}, \p_1, \zeta_1))(\h-\h^\ast) \geq 0
     \end{align*}
     for all $\h \in \mathcal{U}_{ad}$.
     A direct computation of \eqref{derivative lh} leads to the following inequality
    \begin{align}\label{var in}
        \int_{\Sigma'}\h^\ast \cdot(\h - \h^\ast) + \int_{\Sigma} M^\ast\p_1 \cdot(\h - \h^\ast) \, dSdt \geq 0.
    \end{align}
    where $M^\ast$ is the adjoint of $M$ and from [\cite{raymond_feedback}, Lemma 2.4], we have $M= M^\ast$.
    Now substituting the value of $\p_1$ from \eqref{bdry p} in \eqref{var in} we get 
    \begin{align*}
      \int_{\Sigma'} \h^\ast(\h - \h^\ast) \, dSdt - \int_{\Sigma} M\hat{P}\n\cdot(\h - \h^\ast) \, dSdt - \int_{\Sigma} M\frac{\partial\p}{\partial\n}\cdot(\h - \h^\ast) \, dSdt \geq 0, \quad \h \in \mathcal{U}_{ad}.
 \end{align*}
 That completes the proof.
\end{proof}

\begin{remark}
    We can also prove optimality condition \eqref{variational ineqa} in Theorem \ref{1st order necessary optimality condition} from equation \eqref{necessary condition} in Theorem \ref{Optimality Condition by linear}. For this, we take $(\w, \psi)$ to be the unique weak solution of linearized system \eqref{linearise equ} corresponding to boundary $M\Tilde{\eta} = M(\Tilde{\h} - \Tilde{\h^\ast})$, for any $\h \in \mathcal{U}_{ad}$. Then from the optimal condition \eqref{derivative lf}, by taking $(\u_1,u_2)$ as the linearized solution $(\w,\psi)$ and using the fact that $\w|_{\Sigma}= M(\Tilde{\h} - \Tilde{\h^\ast}), \, {\frac{\partial \psi}{\partial n}}|_{\Sigma} = 0$, we can derive
    \begin{align*}
        &\int_{Q}(\u^\ast - \u_{Q})\cdot \w \, dxdt + \int_{Q}(\varphi^\ast - \varphi_{Q})\cdot \psi \, dxdt + \int_{\Omega} (\u^\ast(T) - \u_{\Omega})\cdot \w(T) \, dx \no \\ & + \int_{\Omega} (\varphi^\ast(T) - \varphi_{\Omega})\cdot \psi(T) \, dx = \int_{\Sigma} M^\ast\p_1\cdot(\h - \h^\ast) \, dSdt, \quad \forall \h \in \mathcal{U}_{ad}.
    \end{align*}
    Now from the variational inequality \eqref{necessary condition}, \eqref{variational ineqa} follows easily.
\end{remark}
Finally, using the variational formula \eqref{variational ineqa}, we can interpret the optimal boundary control in terms of the adjoint variables.
\begin{corollary}
 Let $\h^\ast \in \mathcal{U}_{ad}$ be a optimal boundary control associated with $(\mathbf{OCP})$. Then $\h^\ast$ and the adjoint system $(\p, \zeta)$ satisfy the projection formula
 \begin{align}
     \h^\ast = \mathcal{P}_{\mathcal{U}_{ad}}\big(-M(\hat{P}\n +\frac{\partial\p}{\partial\n})\big),
 \end{align}
 where $\mathcal{P}_{\mathcal{U}_{ad}}$ is the orthogonal projector from $\mathrm{L}^2(\Sigma)$ onto $\mathcal{U}_{ad}$.
\end{corollary}

\bibliographystyle{plain}
\bibliography{mybibliography}

\end{document}